\documentclass[11pt]{amsart}

\usepackage[textwidth=16cm, textheight=23cm,marginratio=1:1]{geometry}
\usepackage{tikz-cd}
\usepackage {amssymb}
\usepackage{amsfonts}
\usepackage{amsmath}
\usepackage{amsthm}
\usepackage{tikz-cd}
\usepackage{enumitem}
\usepackage[enableskew]{youngtab}
\usepackage{nicefrac}
\usepackage{ytableau}
\usepackage{xcolor}
\usepackage{tabularray}
\usepackage{nicematrix}
\usepackage{abstract}
\usepackage{booktabs}
\usepackage{float}
\usepackage{nicefrac}
\usepackage{enumitem}
\usepackage{calc}
\usepackage{circuitikz}
\usepackage{xcolor}
\usepackage[T1]{fontenc}

\usepackage[bookmarks,colorlinks=true,%
pdfview=FitH,pdfstartview=FitH,%
urlcolor=blue,pdftitle={Deformations of the connected sum of Gorenstein algebras},%
pdfauthor={Piotr Oszer},%
pdfsubject={},%
pdfcreator={LaTeX},%
pdfkeywords={}
]{hyperref}

\def\A{\mathbb A}

\def\Z{\mathbb Z}
\def\P{\mathbb P}
\def\G{\mathbb G}

\def\F{\mathcal F}
\def\O{\mathcal O}
\def\U{\mathcal U}
\def\m{\mathfrak m}

\def\M{\mathcal M}
\def\X{\mathcal X}
\def\Y{\mathcal Y}

\newcommand{\cactus}[2]{\mathfrak{K}_{#1}\left( #2 \right)}

\DeclareMathOperator{\Mod}{Mod}
\DeclareMathOperator{\coker}{coker}

\DeclareMathOperator{\Tor}{Tor}

\DeclareMathOperator{\mcG}{\mathcal G}
\DeclareMathOperator{\Proj}{Proj}

\DeclareMathOperator{\mcB}{\mathcal B}
\DeclareMathOperator{\Art}{Art}
\DeclareMathOperator{\Set}{Set}
\DeclareMathOperator{\Grp}{Grp}

\DeclareMathOperator{\Spec}{Spec}

\newcommand{\Aut}{ \operatorname{Aut}}

\newcommand{\Hom}{ \operatorname{Hom}}

\DeclareMathOperator{\Hilb}{Hilb}
\newcommand{\HG}[2]{{\Hilb}^{Gor}_{#1}(\A^{#2})}

\DeclareMathOperator{\chr}{char}

\DeclareMathOperator{\Ext}{Ext}
\DeclareMathOperator{\Apolar}{Apolar}
\DeclareMathOperator{\Ann}{Ann}
\newcommand{\Def}{ \operatorname{Def}}
\newcommand{\Der}{ \operatorname{Der}}

\newcommand{\kk}{\Bbbk}
\newcommand{\eps}{\varepsilon}

\newtheorem{lem}{Lemma}[section]
\newtheorem{defi}[lem]{Definition}
\newtheorem{set}[lem]{Setting}
\newtheorem{defi-not}[lem]{Definition/Notation}
\newtheorem{twr}[lem]{Theorem}
\newtheorem{stwr}[lem]{Corollary}
\newtheorem{exm}[lem]{Example}

\newtheorem{remark}[lem]{Remark}
\newtheorem{prop}[lem]{Proposition}

\newtheorem*{twri}{Theorem}

\newtheorem*{notation}{Notation}
\numberwithin{equation}{section}

\newcounter{problem}

\title{Deformations of the connected sum of Gorenstein algebras}
\author{Piotr Oszer}	

\address{Institute of Mathematics, University of Warsaw, Banacha 2, 02-097
	Warsaw, Poland}
\email{po394403@mimuw.edu.pl}
\thanks{The author is partially supported by National Science Centre grant 2023/50/E/ST1/00336. The author would also like to thank Jarosław Buczyński, Joachim Jelisiejew, Feliks Rączka, Andrea Ricolfi and Michał Szachniewicz for helpful discussions and suggestions
}
\begin{document}
\maketitle
\begin{abstract}
	We prove that the Gorenstein locus of the Hilbert scheme of points on $\A^n$ is non-reduced for $n\geq 12$;
	we construct examples of non-reduced points that come from apolar algebras of the sum of general cubics. As a corollary, we get a non-reducedness result for the cactus scheme. We generalise the Białynicki-Birula decomposition to abstract deformation functors, providing a new method of studying deformation theory. Our construction gives us fractal structures on the nested Hilbert scheme.
\end{abstract}
\section{Introduction}
The Hilbert scheme $\Hilb_{m}(\A^n)$ of $m$-points on the affine space $\A^n$ is the moduli space of closed subschemes of length $m$ lying on the affine space. Outside of algebraic geometry, it has various applications, for example, to combinatorics, see \cite{haiman}. The Gorenstein locus $\HG{m}{n}$ is the open locus parameterising the Gorenstein finite subschemes, see \cite[Proposition 2.5]{CGN-07}. 
The global structure of $\HG{m}{n}$ has been the object of intensive study. Irreducibility and smoothness for $n\leq 3$ is part of the folklore, see \cite{kleppe-miro-roig}. It is known that the Gorenstein locus is irreducible 
for $m\leq 13$ and for $m=14$ and $n\leq 5$, see \cite[Theorem A]{CGJ}, which completes a series of papers on this topic \cite{CGN-07,CGN-09,CGN-14,CGN-Gorenstein,CGN-16}.  
In particular, the Gorenstein locus is not dense, since the Hilbert scheme is reducible for $n=3$ and $m\geq 78$, as well as for $n\geq 4$ and $m\geq 8$. This shows that the Gorenstein locus is better behaved than the whole Hilbert scheme.

 There has also been an interest in the question of the existence of nonsmoothable components.
 It has been proven, \cite[Lemma 6.21]{Iarrobino1999}, that $\HG{m}{n}$ is reducible for $n\geq 6$ and $d\geq 14$.
 By a parameter count, it is known that the reducibility also holds for $n=4$, $m\geq 140$ and $n=5$, $m\geq 42$, see \cite[Proposition 6.2]{BuBu-cactus}. \vskip 0.2 cm
 In \cite[Lemma 6.21]{Iarrobino1999} new smooth components of $\HG{2n+2}{n}$ have been constructed for $6\leq n\leq 13$ and not $7$. 
 A recent study \cite[Theorem 1]{Robert-Gorenstein} completed that work by constructing explicit new smooth components of $\HG{2n+2}{n}$ for $n\geq 6$ except for $n=7$. This particular construction fails for $n\leq 5$ and for $n=7$, see \cite{CGN-Gorenstein} and \cite{Bertone-smoohtable}. \vskip 0.2 cm
In this article, we focus on the local structure, especially on the singularity type. Our main result is that the Gorenstein locus is not reduced.
\begin{twr}[{Corollary~\ref{main-result}}] \label{main-result-intro}
		For every $n=12$ or greater than $13$, the Gorenstein locus $\HG{2n+2}{n}$ is not reduced.
	\end{twr}
	It is useful to look at the Gorenstein locus through the lenses of Macaulay duality, hence focusing on apolar algebras, see Definition \ref{def-apolar}.
Our construction is explicit, and the exhibited examples of non-reduced points arise as apolar algebras of the sum of very general cubics of $n_1$ and $n_2$ variables with $n_1+n_2=n$, see Definition~\ref{cubic-very}. The examples provide us embedded points rather than generically non-reduced component, since the locus of algebras with the Hilbert function $(1,2n,2n,1)$ is generically smooth, see Proposition \ref{cubic-component}. Any apolar algebra constructed by taking the sum of two polynomials in disjoint sets of variables is an example of the connected sum of Gorenstein algebras, see Definition~\ref{conn-defi}, it was first introduced in \cite[Section 2. Connected sums]{MR2929675}, the apolar description of the connected sum was developed in \cite{IARROBINO-connected}.
 An immediate corollary is that $\HG{m}{n}$ is non-reduced for $n>10$ for $m$ large enough. The case of $4\leq n\leq 12$ remains open.
\subsection{Application to study of tensors} 
The local structure of the Gorenstein locus is related to the theory of tensors. The main goal of the theory of tensors is to understand the ranks and border ranks of tensors, and
the analogous notions for symmetric tensors. The results in this area have a large variety of applications outside of algebraic geometry, for example, to algebraic statistics \cite{sturm-stat} and complexity theory, see \cite{land-book}, where it was famously applied by Strassen to study bounds on the exponent of the matrix multiplication \cite{strassen-matrix}. The notion of border rank is closely related to the study of secant varieties.
Let $Y\subset \P^N$ be a smooth projective scheme and let $r$ be a fixed positive integer.
We define the secant variety as
$$ \sigma_r(Y):= \overline{\bigcup_{p_1,\ldots, p_r\in Y} \{\langle p_1\ldots p_r \rangle\}} \subset \P^N,$$ 
where $\langle p_1,\ldots, p_r \rangle$ is the linear span. 
In particular, if the scheme $Y$ is the Veronese embedding of $\P^n$, the $\kk$-points of $Y$ represent powers of linear forms. Then, the $\kk$-points of $\sigma_r(Y)$ represent sums of $r$-powers of such linear forms and their limits.
Since it is notoriously difficult to study equations of secant varieties, the cactus varieties $\cactus{r}{Y}$ were introduced as a remedy in a seminal work \cite[Section 2]{BuBu-cactus}. They are defined as
$$ \cactus{r}{Y}:= \overline{\bigcup\{\langle R \rangle |R\subset Y,\ R \text{ is a subscheme of degree at most $r$ }\}} \subset \P^N.$$
The cactus variety carries a natural scheme structure given by minors of catalecticant matrices, see \cite[Theorem 1.5]{BuBu-cactus}. This structure was used in \cite{choi2025singularitiessyzygiessecantvarieties} and in \cite{jabuhanieh-familiar} to show that singularities of the Gorenstein locus of the Hilbert scheme all yield singularities of the cactus scheme.
 Let $\nu_d$ be the $d$-th Veronese embedding.  
In particular, our main result, Corollary \ref{main-result}, implies the following.
\begin{stwr} [{Corollary~\ref{final-app}}]
	Let $n\geq 12$, but not $13$, and $d\geq 4n+4$. The cactus scheme $\cactus{2n+2}{\nu_d(\P^{n})}$ is non-reduced.
\end{stwr}
\subsection{Non-reducedness of the Hilbert scheme}
 The non-reducedness of the Hilbert scheme was exhibited for the Hilbert scheme of a large number of points on $\A^{16}$ in \cite[Corollary 5.1]{JJ-pathologies}, for the Hilbert scheme of $13$ points on $\A^6$ in \cite[Theorem 1.4]{szachniewicz-nonreduce}, for at least $21$ points on $\A^4$ in \cite[Theorem 1.1]{JJA4}  and  for the pointed Hilbert scheme $\Hilb_{1,d}(\A^n)$, for $d\geq 2$ and $n\geq 2$ in \cite[Theorem B]{Unexpected-GGGL}. 

 Let us reformulate, in a uniformed setting, the fundamental technique yielding statements on non-reducedness; this method was developed in \cite{JJ-elementry}, and then applied in \cite{JJ-pathologies,szachniewicz-nonreduce,JJA4,Unexpected-GGGL}.\\
 Let $\M$ be a $\G_m$-scheme and $y\in \M$ a $\G_m$-fixed point. By the Białynicki-Birula decomposition \cite[Theorem 1.1]{ABB_JJLS} there exists a Cartesian diagram
 	$$\begin{tikzcd}
 		\M^{>0} \arrow[d] \arrow[r] & \M^{\geq0} \arrow[d] \arrow[r,"i"]   & \M \\
 		\{y\} \arrow[r]                 & \M^{\G_m}, \arrow[u, bend right] &  
 	\end{tikzcd}$$
 	where $\M_{>0}$ is called the positive spike at the point $y$.
 	Since $y$ is a $\G_m$-fixed point, its tangent space is $\Z$-graded, consequently it admits the weight decomposition
 	$$ T_y\M= (T_y\M)_{<0} \oplus (T_y\M)_{0}\oplus (T_y\M)_{>0}.$$ 
 \begin{twri}[{\cite[Proposition 1.6]{ABB_JJLS}}]
 	Assume that the negative tangent space $(T_y\M)_{<0}$ vanishes.
 	Then the map
 $i \colon \M^{\geq 0} \to \M$ 
 is an open immersion near $i(y) \in \M$.
\end{twri}
 In particular, if the negative tangents vanish at the chosen $\G_m$-fixed point $y$, then the non-reducedness of $\M^{\geq 0}$ yields a non-reduced structure of $\M$.
 Since the map $\M^{\geq 0}\to \M^{\G_m}$ admits a retraction, the non-reducedness of the positive spike at $y$ implies non-reducedness of $\M^{\geq 0}$, for details see Lemma \ref{lemma-section2}. \vskip 0.2 cm
 However, in the Gorenstein case, the vanishing of the negative tangents rarely happens. For example it is open if there is a Gorenstein local Artinian algebra in four variables with only trivial negative tangents, see \cite[Question 1.7]{Staal}. In the known cases, apolar algebras of cubics or quartics of maximal Hilbert function, this condition implies smoothness of the positive spike \cite[Lemma 6.21]{Iarrobino1999}, \cite[Theorem 3.1]{EliasRossi15},  \cite[Theorem 3.1]{EliasRossi12}, \cite[Proposition 1.1]{JJ-class-gor}. Hence the methods cannot be applied to prove non-reducedness. \vskip 0.2 cm 
 To overcome this problem, we develop a new method which refines the one described above. We introduce the Białynicki-Birula decomposition for good deformation functors, which can be thought as the infinitesimal version of the Białynicki-Birula decomposition for algebraic stacks, see Definition \ref{defi-good}.
 The main reference for deformation theory is \cite{sernesi_Deformations_of_algebraic_schemes} and a foundational paper \cite{Fantechi-Manetti}.
  Let $X$ be an affine $\G_m$-scheme.
  The functor $\Def(X)$ admits an associated triple 
   	$$T^{0}_X, \ T^{1}_X, \ T^{2}_X,$$
   	consisting of the infinitesimal automorphism group, the tangent space, and an obstruction space.
   	The vector spaces $T^i_X$ are Lichtenbaum-Schlessinger functors, defined in {\cite[3.~The~$T^i$~Functors]{hardef}}. The $\G_m$-action on $X$ induces an action on the functor and yields graded structures on each of $T^i_X$ vector spaces.
   Then $\Def^{\geq 0}(X)$ parametrises deformations of $X$ together with a filtration on the structural algebra, see Definition \ref{defi-filt-fun}.
 The filtered deformation functors have formal properties, analogous to the Białynicki-Birula decomposition, for example the following.
 \begin{twr} [{Theorem \ref{+obs}}]
 	Let $X$ be an affine $\G_m$-scheme
 	The functor $\Def^{\geq 0}(X)$ is a good deformation functor with the associated triple (infinitesimal automorphisms, tangent space, obstruction space)
 	$$(T^{0}_X)_{\geq 0}, \ (T^{1}_X)_{\geq 0}, \ (T^{2}_X)_{\geq 0}.$$
 \end{twr}
 
 The natural morphism $\Def^{\geq 0}(X)\to \Def^{0}(X)$, induced by taking the associated graded algebra given by the filtration, admits a section $\Def^0(X)\to \Def^{\geq 0}$. We also define an abstract equivalent of the positive spike at a $\G_m$-fixed point $\Hilb^{>0}(X)$. This is the functor of strictly negative deformations, see Definition \ref{abs-fibre}, given by the Cartesian diagram
	$$\begin{tikzcd}
	\Def^{>0}(X) \arrow[d] \arrow[r] & \Def^{\geq0}(X) \arrow[d] \arrow[r] & \Def(X) \\
	* \arrow[r]                      & \Def^{0}(X) \arrow[u, bend right].            
\end{tikzcd}$$
The key technical part of the results is using the abstract Białynicki-Birula decomposition combined with
the theory of so-called Braids of $T^i$-functors, first introduced by Buchweitz {\cite[2.4.3.2 page~69]{Buch}}.
Our techniques do not use the assumption of finiteness of the scheme, nor that it is Gorenstein; instead, we rely on purely cohomological invariants,  where the crucial condition is the vanishing of the tangent space in degrees~$<-1$. Henceforth, they should be useful in other contexts.
 \subsection{Fractal family}
 Apart from proving non-reducedness, we describe the structure of the completed local ring at the described non-reduced points of the Gorenstein locus. The heart of the description lies in the notion of \textit{a fractal family}, which we introduce in Definition \ref{fractal}. The \textit{fractalness} is heavily influenced by fractal-like families introduced in \cite[Section 3.3]{szachniewicz-nonreduce}.
A pointed family of schemes is called \emph{a fractal family} if a local ring at the distinguished point is smoothly equivalent to a deformation of the special fibre.
Intuitively, this means that the base of the pointed family naturally contains a copy of the special fibre, yielding a non-trivial deformation of the fibre over itself. 
 	\begin{lem}[Lemma {\ref{frac-nonred}}]
 	Let $\U\to \M$ be a fractal family at a $\kk$-point $x \in \M$. Then the $\kk$-point $x$ is a non-reduced point of $\M$.
 \end{lem}
Using this notion, we generalise \cite[Theorem 2.8]{Unexpected-GGGL}, where the authors have shown that the pointed Hilbert scheme $\Hilb_{1,m}(U)$ is non-reduced for $U$ a smooth quasi-projective variety and $m\geq 2$ a positive integer.
\begin{twr} [{Theorem~\ref{fractal1}}]
		Let $U$ be a smooth quasi-projective variety and let $m\geq 2$ be a positive integer. Then, the nested Hilbert scheme $\Hilb_{1,m}(U)$ admits a fractal structure at any point ${[*,X]\in \Hilb_{1,m}(U)}$, such that $[X]\in \Hilb_m(U)$ and $X$ is uncleavable.
\end{twr}
In Lemma \ref{fractal2} we show that the nested Hilbert scheme $\Hilb_{m-1,m}(U)$ admits a fractal structure at pairs $(X\#\emptyset, X)$, where $X$ is the spectrum of a local Gorenstein algebra with vanishing $(T^1_X)_{<-1}$.
The subscheme $X\#\emptyset$ is given as the vanishing locus of the socle of $\O_X$. This way, in Corollary~\ref{szach-theorem}, we reprove the result \cite[Theorem 1.5]{szachniewicz-nonreduce} describing the deformations of $X\#\emptyset$.
As showcased above, this article unifies the known methods employed in the study of local properties of the Hilbert scheme and develops them further.
\subsection{Overview of the main points of argument}

The pointed schemes studied in Section \ref{sec-pointed} serve as building blocks for further constructions. In Section \ref{sec-glued}, we glue the schemes along the distinguished points
$$\begin{tikzcd}
	* \arrow[d] \arrow[r] & X \arrow[d] \\
	Y \arrow[r]           & X\cup_{*} Y,
\end{tikzcd}$$
and obtain an intermediate result, of independent interest, on the negative deformations of the union along the point.
\begin{twr} [{Theorem \ref{step2-main}}]
	Assume that $X,Y$ are pointed finite $\G_m$-schemes which are spectra of $d$-level algebras, see Definition \ref{defi-level}, and assume that neither is a reduced point.
	Then the morphism
	$$ \Def^{<0}(* \to X) \times \Def^{<0}(* \to Y) \to \Def^{<0}(X\cup_* Y)$$
	is smooth and induces an isomorphism on the tangent spaces.	
\end{twr}
Next, we assume that both $X$ and $Y$ are Gorenstein. We take the union of $X$ and $Y$ along the point, then identify the socle of $X$ with the socle of $Y$. This way, we obtain the connected sum $X\#Y$, see Definition \ref{conn-defi}.
Such a scheme is Gorenstein and leads up to the main result, Theorem \ref{main-result-intro}.

\section{Preliminaries}
Throughout this paper, we fix a positive integer $m$ and an algebraically closed field $\kk$ of characteristic either $0$ or co-prime with $m$.
\subsection{Macaulay's inverse systems} \label{apolar-section}
To fix notation, we introduce the basics of apolar algebras; for a more thorough exposition, see~{\cite[Section 2.3]{JJ-class-gor}}.
\vskip 0.2 cm
Let us take a pair of polynomial algebras $S_x=\kk[{x_1},\ldots ,x_{n_1}]$, $S_y=\kk[{y_1},\ldots, {y_{n_2}}]$ and their tensor product $S=\kk[{x_1},\ldots ,{x_{n_1}},{y_1},\ldots, {y_{n_2}}]$. \\
Denote by $P_x$ (respectively $P_y$, $P$) a divided power series ring in variables $X_1,\ldots, X_{n_1}$ (resp. in variables $Y_1,\ldots, Y_{n_2}$, and $X_1,\ldots, X_n,Y_1,\ldots, Y_{n_2}$).
There is an action of $S_x$ on $P_x$ by \textit{contraction}:
$$x^{\mathbf a}\circ X^{\mathbf b} = \begin{cases}
 X^{\mathbf b -\mathbf a} \text{ if }  \mathbf b \geq \mathbf a, \text{ that is } b_i\geq a_i\text{ for } i=1,\ldots m;\\
 0 \ \text{otherwise.}
\end{cases}$$
Here $\mathbf a = (a_1,\ldots, a_{n_1})$ and $x^{\mathbf a} = x_1^{a_1}\ldots x_n^{a_{n_1}}$.
This action is extended by $\kk$-linearity on $S_x$ and on $P$.
\begin{defi}\label{def-soc}
	Let $(A,\m)$ be a local $\kk$-algebra, then \emph{the socle of $A$} is the annihilator of $\m$.
\end{defi}
\begin{defi}\label{def-apolar}
	Let $F\in P$ be a polynomial, then we define the apolar algebra to be:
	$$\Apolar(F):=\frac{S}{\Ann(F)}=\frac{S}{(s\in S: s\circ F=0)}.$$
	The socle of $\Apolar(F)$ is the ideal generated by an element $f\in S$, such that $f\circ F=1$.
\end{defi}
	Let us take a Gorenstein zero-dimensional subscheme $Y$ of $\A^n=\Spec(S)$ supported at the origin, then $Y=\Spec(S/I)$.
	Such algebras $S/I$ are characterised by the existence of a polynomial $F \in P$, such that $I=\Ann(F)$ and $S/I=\Apolar(F)$. Such a polynomial $F$ is not unique. For details, see \cite[Chapter 21]{eisenbud}.
	\begin{exm}
		Assume that the characteristic of the field $\kk$ is not equal to $2$.
		Let \[F=X_1X_2X_3+X_1X_2^2\in P_x, \] then $\Ann(F)=(x_3^2,x_2^2-x_2x_3,x_1^2).$
	\end{exm}
	
	\begin{defi}\label{HGor-defi}
		Let $\HG{m}{n}$ denote the Gorenstein locus of the Hilbert scheme of $m$-points on $\A^n$. 
	\end{defi}
	The Gorenstein locus is open in the Hilbert scheme $\Hilb_{m}(\A^n)$, see \cite[Proposition 2.5]{CGN-07}.
	However, in general it is not dense for example for $n=3$ and $m\geq 78$, see \cite[Example 4.3]{Iar-Compressed}.

\subsection{Deformation theory}

In this subsection, we introduce fundamental notions of the classical deformation theory.
 The main reference for this section is \cite{sernesi_Deformations_of_algebraic_schemes} and \cite{Fantechi-Manetti}. 
 
The category of local Artinian $\kk$-algebras with the residue field $\kk$ is denoted by $\Art_\kk$.
Let us take a complete local ring $R$, then we can define a functor $h_R\colon \Art_\kk \to \Set$ by setting
$$ h_{R}(A)= \Hom(R,A).$$
Note that we recover the tangent space to $\Spec(R)$ by taking the value of $h_R$ on $\kk[\eps]/\eps^2$. In general values of the functor are thought to be \textit{infinitesimal deformations}.
Any functor $\Art_\kk\to \Set$ isomorphic to $h_{R}$ for some complete local ring $R$ is called \emph{prorepresentable}, for a reference see \cite[2.2 Functors of Artin rings]{sernesi_Deformations_of_algebraic_schemes}. \\
A \emph{small extension} is a surjection in $\Art_\kk$
$$ A'' \xrightarrow{f} A,$$  
such that the kernel of $f$ is annihilated by the maximal ideal of $A''$.
It can be seen that $h_R$ can be recovered from values on small extensions. Deformation theory is the study of functors arising from a geometrical context and \textit{mimicking} the behaviour of the prorepresentable one, where the value on $A\in \Art_{\kk}$ is the set of deformations over $A$.
 \vskip 0.2 cm
\begin{defi-not}
	A functor of Artin rings is a covariant functor $F\colon \Art_\kk \to \Set$, such that $F(\kk)=*$.
\end{defi-not}

\begin{defi}[{\cite[Schlessinger conditions, Theorem 2.11]{schlessinger}}]
	Let $F$ be a functor of Artin rings. 
	Let $A'\to A$ and $A''\to A$ be a pair of homomorphisms in $\Art_\kk$, and consider the natural map
	\begin{equation} \label{schless-eq}
		F(A'\times_A A'' ) \to F(A')\times_{F(A)} F(A'')
	\end{equation}
	Then the following conditions are the Schlessinger conditions:
	\begin{enumerate}[label=(H\arabic*)]
		\item\label{H1} if $A''\to A$ is a small extension then the map \eqref{schless-eq} is surjective.
		\item\label{H2}  whenever $A=\kk$ and $A''=\kk[\eps]/\eps^2$, then the map \eqref{schless-eq} is bijective.
		\item\label{H3}  The conditions \ref{H1} and \ref{H2} are satisfied, and $F(\kk[\eps]/\eps^2)$ is a finite dimensional vector space, see Remark \ref{rem-h3}.
		\item\label{H4}  if $A''\to A$ is a small extension then the map \eqref{schless-eq} is bijective.
	\end{enumerate}
\end{defi}

\begin{remark}\label{rem-h3}
	By \cite[Lemma 2.12]{Fantechi-Manetti} the conditions \ref{H1}-\ref{H2} imply that $F(\kk[\eps]/\eps^2)$ is a $\kk$-vector space, called the tangent space to $F$ (and denoted $t_F$).
\end{remark}

\begin{defi}[{\cite[Definition 2.2.4]{sernesi_Deformations_of_algebraic_schemes}}]\label{defi-smooth}
	We say that a morphism of $F\to G$ of functors of Artin rings is \emph{smooth} if for every surjection $A'\to A$ in $\Art_\kk$ the natural map:
	$$ F(A')\to F(A)\times_{G(A)} G(A')$$
	induced by the commutative diagram
	$$ \begin{tikzcd}
		F(A') \arrow[d] \arrow[r] & G(A') \arrow[d] \\
		F(A) \arrow[r]            & G(A)           
	\end{tikzcd}
	$$
	is surjective.
\end{defi}

\begin{defi}[{\cite[Definition 2.2.6]{sernesi_Deformations_of_algebraic_schemes}}]\label{defi-semi}
	Let $F$ be a functor of Artin rings. We say $F$ admits \emph{a semi-universal family} if there is a smooth morphism $H\to F$ from a prorepresentable functor $H$, such that the induced map of the tangent spaces is bijective.
\end{defi}

\begin{twr}[{\cite[Theorem 2.11]{schlessinger}}] \label{schless-twr}
	If $F$ satisfies the conditions \ref{H1},\ref{H2},\ref{H3}, then it admits a semi-universal family. Furthermore, if it satisfies \ref{H4} then $F$ is prorepresentable.
\end{twr}

\begin{defi}\label{defi-good}
	A functor of Artin rings $F$ is \emph{a good deformation functor} if it satisfies the conditions \ref{H1} and \ref{H2}, and also the condition of linearity of obstructions, see \cite[Definition~2.9]{Fantechi-Manetti}.
\end{defi}

We now summarise the crucial properties of the good deformation functors.

\begin{twr}[{\cite[Lemma 2.12, Theorem 6.11]{Fantechi-Manetti}}]
	Let $F$ be a good deformation functor. Then $F$ admits a natural tangent-obstruction theory, that is, a pair of $\kk$-vector spaces $(t_F,o_F)$ with the obstruction space $o_F$.
	If $F\to G$ is a map of good deformation functors, then it induces the linear maps on the tangent and obstruction spaces.
\end{twr}
For a more comprehensive exposition of obstruction theories, see \cite[2.1 Obstructions]{sernesi_Deformations_of_algebraic_schemes}. 
\begin{defi}\label{unob}
	Let $f\colon F\to G$ be a morphism of good deformation functors. We say that $f$ is \emph{unobstructed} if the induced map between obstruction spaces is injective.
\end{defi}	
\begin{remark}
	The choice of linear obstruction space for a good deformation functor is not unique. Unobstructness should be understood as follows: there is a choice of obstruction spaces such that the induced linear map between them is injective. 
\end{remark}

\begin{exm}
	Let us consider a trivial deformation functor $*$ prorepresented by $\kk$
	\begin{align*}
		*:=h_{\kk}\colon \Art_\kk \to \Set,\\
		A \mapsto \Hom_{\kk}(A,\kk).
	\end{align*}
	The tangent space is zero \-dimensional. The minimal obstruction space is $0$. 
	 But we can choose any other vector space $V$ as the obstruction space.
	We take a good deformation functor $F$ with a chosen tangent-obstruction theory $(t_F,o_F)$.
	Let us take a map $F\xrightarrow{\phi} *$, then the map of obstruction spaces is $o_F \to V$ is $0$.
	Then $\phi$ is unobstructed if and only if we can choose $o_F$ to be $0$, then we say that $F$ is a smooth functor. 
	 On the other hand, any map $*\to F$ is always unobstructed, even though $V\to o_F$ is a zero map.
\end{exm}

\begin{twr}[{\label{fantechi-manetti}\cite[Proposition 2.3.6]{sernesi_Deformations_of_algebraic_schemes}, \cite[Lemma 6.1]{Fantechi-Manetti}}]
	Let $f\colon F\to G$ be an unobstructed morphism of good deformation functors. Assume that the induced map of the tangent spaces is surjective. Then $f$ is smooth.
\end{twr}

	In the latter parts of the paper, we will employ smooth equivalence of deformation functors, see \cite{vakil-law}, \cite[Pointed schemes and smooth equivalence]{JJ-pathologies}. 
	
\begin{defi}
	We say that functors $G$ and $F$ are smoothly equivalent if there exists a functor of Artin rings $H$ with smooth maps
	$$\begin{tikzcd}
		& H \arrow[ld, "sm"'] \arrow[rd, "sm"] &   \\
		F &                                             & G.
	\end{tikzcd}$$
\end{defi}

\begin{defi}
	Let $F\to G$ be a map of good deformation functors. Then \emph{a relative tangent-obstruction theory} is a pair of vector spaces $t_r$ and $o_r$ together with maps forming an exact sequence of linear spaces
	$$0 \to t_r\to t_F \to t_G \to o_r \to o_F \to o_G.$$
\end{defi}

\begin{lem} \label{technical-deformation} 
	Let us take a diagram of deformation functors
	$$\begin{tikzcd}
		F \arrow[d] \arrow[r] & G \arrow[d] \\
		* \arrow[r]           & H .         
	\end{tikzcd}$$
	Assume that the diagram above induces the following exact sequence of vector spaces,
	$$t_F\to t_G \to t_H \xrightarrow{0} o_F \to o_G \to o_H. $$
	Then the induced map $F\xrightarrow{\phi} P:= G\times_{H} * $ is smooth.
\end{lem}
\begin{proof}
	The exactness of the sequence implies that $\phi$ is surjective on the tangent spaces. The map
	$$o_F \to o_G$$ is injective and factors through $o_P$, thus we get that the map $\phi$ is unobstructed. By Theorem~\ref{fantechi-manetti}, we get the conclusion.
\end{proof}
Now we state two very useful lemmas providing us a simple technique in proving non-reducedness.
\begin{lem}\label{lemma-section}
	Let us take a Noetherian $\kk$-algebras $R$ and $S$.
	Let $\phi\colon R \to S$ be a homomorphism. Assume that the induced morphism of schemes $\phi^\#\colon \Spec S\to \Spec R$ has all fibres irreducible and that $\phi^\#$ admits a section $\psi^\#$.
	Let us take a $\kk$-point of $\Spec R $, then we get a Cartesian diagram
	$$\begin{tikzcd}
		\Spec S\otimes_R \kk \arrow[d] \arrow[r] & \Spec S \arrow[d, "\phi^\#"']             \\
		\Spec \kk \arrow[r]                      & \Spec R. \arrow[u, "\psi^\#"', bend right]
	\end{tikzcd}$$
	Assume that the fibre $S\otimes_R \kk$ is finite, then either $\phi$ is locally an isomorphism around the chosen point or $S$ is non-reduced. 
\end{lem}
\begin{proof}
	First, we prove that the homomorphism $\phi$ induces a bijection between the irreducible components. 
	Let us take the generic point $\eta$ of an irreducible component $V$ of $\Spec(R)$. The fibre of $\phi$ over $\eta$ contains of at least one generic point of $\Spec(S)$ since $\phi$ is surjective. And at most one since any two components that are mapped to $V$ intersect along the image $\psi^{\#}(V)$. 
	\vskip 0.2 cm
	Since the statement has a local character we can localize at the chosen closed point and assume that $S$ is local. Let $\mathfrak{p} \subset S$ be a minimal prime ideal. As shown above the homomorphism $\phi$ induces a morphism of irreducible components
	$\phi_{\mathfrak p} \colon R/ \mathfrak p R \to S/\mathfrak p $, which is a section of 
	$\psi_{\mathfrak p}  \colon S/\mathfrak p \to R/\mathfrak p R$.
	By \cite[Theorem 10.10]{eisenbud} the homomorphism $\phi_{\mathfrak p}$ induces the inequality
	$$\dim S/\mathfrak p \leq \dim R/\mathfrak p R + \dim S/\mathfrak p\otimes_R\kk= \dim R/\mathfrak p R.$$
	Since $\psi_{\mathfrak p}$ admits a section, it is surjective, therefore $\dim R /\mathfrak p R \leq \dim S /\mathfrak p$. Hence $\dim R/\mathfrak p R =\dim S/\mathfrak p$.
	Since $S/\mathfrak p$ is a domain, this implies that the kernel of $\psi_{\mathfrak p}$ is the zero ideal.
	Now let us take $I$ to be the kernel of $\psi$. We have shown that $I\subset \mathfrak p$ for any minimal prime ideal $\mathfrak p$, hence $I$ is contained in the nilradical of $S$.
	If $S$ is reduced then $\psi$ is an isomorphism, and $\phi$ as an inverse is also an isomorphism.
\end{proof}
	Now we replicate the proof of Hedgehog point theorem \cite[Theorem 3.1]{szachniewicz-nonreduce}.
	\begin{lem}\label{lemma-section2}
		Let $X$ be a $\G_m$-scheme.
		Assume that there is a $\G_m$-equivariant affine morphism ${\pi \colon X\to X^{\G_m}}$.
		Let $x\in X^{\G_m}$ be a closed point, and assume that fibre of $\pi$ over $y$ is finite and non-reduced. Then $X$ is non-reduced.
	\end{lem}
	\begin{proof}
		Since $\pi$ is $\G_m$-equivariant, then the natural embedding $X^{\G_m} \subset X$ is a section of $\pi$.
		Since $\pi$ is affine, we get $X$ as a relative spectrum of graded algebra $R$ over $\O_{X^{\G_m}}$.
		We can consider projectivization and get a projective morphism ${\Proj_{\O_{X^{\G_m}}}(R) \to {X^{\G_m}}}$.
		Since the fibre over $y$ is finite, by the semicontinuity of dimension of fibres  \cite[\href{https://stacks.math.columbia.edu/tag/0D4I}{Tag 0D4I}]{stacks-project}  the fibres of $\phi$ are finite in the open neighbourhood of $x$ in ${X^{\G_m}}$. Any such fibre is finite $\G_m$-scheme with a unique fixed point hence it is local, in particular irreducible. Hence by Lemma \ref{lemma-section} we get non-reduceness.
	\end{proof}

 	The technique above provides us with results on reducedness of complete local algebras, equivalently on prorepresentable deformation functors. In general, we do not assume as much about deformation functors. That motivates the next definition.
 \begin{defi}
 	Let $F$ be a good deformation functor satisfying condition \ref{H3}. Then $F$ admits a semi-universal family. We say that $F$ is non-reduced, if its semi-universal family is prorepresented by a non-reduced algebra.
 \end{defi}
 Apart from the tangent space and obstruction space, in the case of deformation functors "coming from geometry", we get the third important space, the one controlling automorphisms of deformations.

 \begin{defi}[{\cite[Example 7.2]{Fantechi-Manetti}}]\label{aut}
 	Let $X$ be a scheme.
 	We define \emph{the automorphism functor} by setting:
 	\begin{align*}
 		&\Aut_{X}\colon \Art_{\kk} \to \Grp \\
 		&\Aut_{X}(A) =  \text{the group of automorphisms of deformation $X_A:=X\times \Spec(A)$} ,
 	\end{align*}
 	where an automorphism of a deformation is an automorphism of $\Spec(A)$-scheme $X_A$ over the identity on $X$.
 \end{defi} 
 
 \begin{lem}[{\cite[3.11]{schlessinger}}, {\cite[Example 7.2]{Fantechi-Manetti}}]\label{defi-t0}
 	The automorphism functor is a good deformation functor.
 	The tangent space to $\Aut_X$ is called \emph{the group of infinitesimal automorphisms}.
 \end{lem}
 
 \begin{exm}[{\cite[Lemma 1.2.6]{sernesi_Deformations_of_algebraic_schemes}}]
 	In the case of $X$ being an affine scheme $\Spec(B)$, its group of infinitesimal automorphisms is canonically isomorphic to the module of derivations $\Der_\kk(B,B)$.
 \end{exm}
 	From now on, instead of a pair consisting of the tangent and an obstruction space, for each local deformation problem we are considering, we are going to associate a triple.
 \begin{defi}\label{triple}
 	 For each local deformation problem with a good deformation functor $F$ \emph{an associated a triple $(\mathfrak a_F, t_F,o_F)$}, consists of the group of infinitesimal deformations, the tangent space, and \textbf{an} obstruction space.
 \end{defi}
 
 \begin{remark}
 	Note that in Notation \ref{triple}, a triple is not associated to a functor but to a deformation problem. A functor of Artin rings takes values in sets, hence the functor does not provide any intrinsic information on automorphisms of parametrised objects. Such data is typically encoded in a stack, roughly by taking values in the category of groupoids. However, in this article, we are not using the language of stacks and stick to the classical notions. This is a deliberate choice made for the sake of clarity of the exposition. 
 \end{remark}

 \subsection{Buschewitz braid}\hfill\\
 The crucial technique of this paper is a method that enables us to compare varying deformation functors coming from a morphism of schemes $X\to Y$. The main content of this subsection is the statement of a result of Buchweitz \cite[2.4.3.2, Page 69]{Buch}, which shows interdependencies of several deformation functors coming from a given morphism. \vskip 0.2 cm
 
 Let $Y$ be a scheme and $\mathcal F$ a quasi-coherent sheaf on $X$ and $\O_X$ a $\O_Y$-algebra.
 Recall Lichtenbaum-Schlessinger functors $T^i_{\O_Y}(\O_X,F)$ defined for $i=0,1,2$, see \cite[Construction~3.1]{hardef}.
 \begin{defi}[{\cite[Section 3.4]{sernesi_Deformations_of_algebraic_schemes}, \cite[Appendix C.5]{Greuel-book}}] \label{defi-XY}
 	Let $f\colon X\to Y$ be a morphism of schemes.
 	Let $A\in \Art_\kk$. We define the following good deformation functors.
	 \begin{itemize}
	 	\item \cite[Definition 3.4.1]{sernesi_Deformations_of_algebraic_schemes} Deformations of a morphism. \vskip 0.2 cm
	 	$\Def(X\to Y)(A)$ is a set of isomorphism classes of Cartesian diagrams 
	 \begin{equation}	 \label{mor-diag}	
	\begin{tikzcd}
	 		X \arrow[d, "f"] \arrow[r] & \X \arrow[d, "\tilde f"]    \\
	 		Y \arrow[d] \arrow[r]      & \Y \arrow[d, "\psi"] \\
	 		* \arrow[r]                & \Spec(A),           
	 	\end{tikzcd}	 \end{equation}	
	 	where $\psi$ and $\psi\circ \tilde f$ are flat morphisms.
	 	\vskip 0.2 cm
	 	The natural associated triple for that problem is denoted by $T^i_{X\to Y}$, for $i=0,1,2$. For a reference see \cite[Theorem C.3.6.]{alper-notes}.
	 		 	\vskip 0.2 cm
	 	\item \cite[Subsection 3.4.2]{sernesi_Deformations_of_algebraic_schemes} Deformations of a morphism with target fixed. \vskip 0.2 cm $\Def(X/ Y)(A)$ is the set of isomorphism classes of Cartesian diagrams as in Diagram~\eqref{mor-diag} but $\Y=Y\times \Spec(A)$. \vskip 0.2 cm
	 	The natural associated triple for that problem is $$T^i_{X/ Y}:=T^i_{\O_Y}(\O_X,\O_X),$$ for $i=0,1,2$.
	 		 	\vskip 0.2 cm
	 	\item Deformations of a morphism with source fixed. \vskip 0.2 cm $\Def(X\backslash Y)(A)$ is the set of isomorphism classes of Cartesian diagrams as in Diagram~\eqref{mor-diag} but $\X=X\times \Spec(A)$.
	 	The natural associated triple for that problem is denoted ${T^i_{X\backslash Y}}.$
	 	In case of the map $f$ being a closed embedding, we get
	 	$${T^i_{X\backslash Y}=T^i_\kk(\O_Y,I)},$$ for $i=0,1,2$, where $I$ is the kernel of the induced map $\O_Y\to f_*\O_X$.
	 	\vskip 0.2 cm
	 	\item \cite[Subsection 3.4.1]{sernesi_Deformations_of_algebraic_schemes} Deformations of a morphism with target and source fixed. \vskip 0.2 cm
	 	 $\Def(X\backslash X \to Y/ Y)(A)$ is a set of isomorphism classes of Cartesian diagrams as in Diagram \eqref{mor-diag} but $\X=X\times \Spec(A)$ and $\Y=Y\times \Spec(A)$.
	 	 The natural associated triple for that problem is 
	 	 $$(0, \ T^1_{X\backslash X \to Y/ Y,} \ T^2_{X\backslash X \to Y/ Y}),$$
	 	 where $${T^i_{X\backslash X \to Y/ Y}:=T^{i-1}_\kk(\O_Y,\O_X)},$$ for $i=1,2$.
	 \end{itemize} \end{defi}
 \begin{defi}\label{rmk-loc-hilb}
 	If $X\to Y$ is a closed embedding then, the functor of deformations of $X\to Y$ with the fixed target is also known as the local Hilbert functor.
 \end{defi}
 There are several maps between the deformation functors of maps, as presented in the following diagram
$$\begin{tikzcd}
	&                                          & \Def(X\backslash X \to Y/Y) \arrow[ld] \arrow[rd]             &                                &         \\
	\Def(Y) & \Def(X\backslash Y) \arrow[rd] \arrow[l] &                                                               & \Def(X/Y) \arrow[ld] \arrow[r] & \Def(X) \\
	&                                          & \Def(X \to Y) \arrow[rru, bend right] \arrow[llu, bend left] &                                &        
\end{tikzcd}$$
 Those maps induce several long exact sequences that are interdependent. This was first observed by Illusie in \cite{Illusie} and elegantly put into a commutative diagram (later called a Braid) by Buchweitz in \cite[2.4.3.2 page 69]{Buch}, for further reference, see \cite[Appendix C.5]{Greuel-book}.
 \begin{equation}\label{braid0}	\tag{$\spadesuit$}
 	\begin{tikzcd}[sep=small]
 			0 \arrow[dd, bend right=49, Rightarrow ]  & & 0 \arrow[dd, bend left=49, dotted ]  \\
 		& 0 \arrow[ld, dashed] \arrow[rd] & \\
 		T^0_{X\backslash Y} \arrow[rd, Rightarrow] \arrow[dd, bend right=49, dashed ]                      &                                                           & T^0_{X / Y} \arrow[ld, dotted] \arrow[dd, bend left=49] \\
 		& T^0_{X\to Y} \arrow[rd, Rightarrow] \arrow[ld, dotted]    &                                                                              \\
 		T^0_Y \arrow[dd, bend right=49,dotted] \arrow[rd, dashed]                       &                                                           & T^0_X \arrow[dd,Rightarrow, bend left=49] \arrow[ld]         \\
 		& T^1_{X\backslash X \to Y/Y} \arrow[rd, dashed] \arrow[ld] &                                                                   \\
 		T^1_{X/ Y} \arrow[rd, dotted] \arrow[dd, bend right=49]             &                                                           & T^1_{X \backslash Y} \arrow[ld, Rightarrow] \arrow[dd, dashed, bend left=49] \\
 		& T^1_{X\to Y} \arrow[rd, dotted] \arrow[ld, Rightarrow]    &                                                                   \\
 		T^1_X \arrow[dd, Rightarrow, bend right=49] \arrow[rd]                       &                                                           & T^1_Y \arrow[dd, dotted, bend left=49] \arrow[ld, dashed]         \\
 		& T^2_{X\backslash X \to Y/Y} \arrow[ld, dashed] \arrow[rd] &                                                                   \\
 		T^2_{X\backslash Y} \arrow[rd, Rightarrow] \arrow[dd, dashed, bend right=49] &                                                           & T^2_{X/Y} \arrow[ld, dotted] \arrow[dd, bend left=49]             \\
 		& T^2_{X\to Y} \arrow[ld, dotted] \arrow[rd, Rightarrow]    &                                                                   \\
 		{T^2_Y}                                                                           &                                                           & {T^2_X}.                                                               
 	\end{tikzcd}
 \end{equation}

\section{Filtered deformations}\label{sec-3}
	In this section, we construct a generalisation of Białynicki-Birula decomposition for (good) deformation functors. The Białynicki-Birula decomposition for a scheme is a fruitful tool in study of the Hilbert scheme, see \cite{ES1,JJ-elementry,JJ-pathologies}.
	First, we construct the local functor, Definition~\ref{defi-filt-fun}, of filtered deformations of graded algebra, then establish that it is a good deformation functor Theorem \ref{filt-good} and identify its associated triple, Theorem \ref{+obs}. Next, we compare the functor with the Hilbert functor, Lemma \ref{smooth-nonstrict}, Lemma \ref{abb-smooth}. Finally, we define another functor, analogous to the fibre of the Białynicki-Birula decomposition over a fixed point, Definition \ref{abs-fibre}, Corollary~\ref{fund-negative-def}, Corollary~\ref{smooth-strict}. We end the section with comments on the various variants of functors of filtered deformations of maps of schemes.\vskip 0.2 cm

	Let us fix $X$ to be an affine scheme $\Spec B$.
	
	\begin{defi}\label{filt-alg-defi}
		Let $A$ be a $\kk$-algebra.
		A \emph{filtered $A$-algebra} $C$ is a pair $(C,F)$, where $C$ is an $A$-algebra and
		$F$ is a filtration by $A$ sub-modules of $C$
		$$\ldots \supset\ldots \supset F_{-1} \supset F_0 \supset F_1 \supset \ldots, $$ such that for every $i,j\in \Z$, we have $F_iF_j \subset F_{i+j}.$ 	\vskip 0.2 cm
		\emph{A homomorphism of filtered algebras }$(C,F)$ and $(C',F')$ is a homomorphism of algebras
		$\phi \colon C\to C'$ such that $\phi(F_i)\subset F'_{i}$.
	\end{defi}
	
	\begin{defi}\label{two-filtrations}
		Let $X=\Spec B$ be an affine $\kk$-scheme with a $\G_m$-action, that is, $B$ is $\Z$-graded. 
	We can consider two $\Z$-grading structures on $B=\bigoplus B_i$, the given one, and the opposite one $B=\bigoplus B_{-i}$. \vskip 0.2 cm
	Then, $B$ admits { the filtration by vanishing order}
	$$F_i= \bigoplus_{j\geq i} B_{j} $$ 
	and { the filtration by degree}
	$$F_i= \bigoplus_{j\geq i} B_{-j}. $$ 
	In this way we have two natural filtered algebras associated to any graded algebra $B$.
	\end{defi}
	
	
	\begin{defi}\label{filt-alg-flat-defi}
		A filtered $A$-algebra $(C,F)$ is flat if all elements of the filtration are flat $A$-modules.
	\end{defi}

	\begin{defi}\label{defi-filt-fun}
		Let $X$ be a spectrum of $\Z$-graded $\kk$-algebra $(B,F)$ of finite type.
		We define a functor of \emph{filtered deformations of $X$} to be the functor of Artin rings 
		$$\Def^{\geq 0}(X)\colon \Art_\kk\to \Set$$
		which sends $(A,\m)\in \Art_\kk$ to the set of isomorphism classes of pairs, a deformation o $X$ over $\Spec(A)$ 
		$$\begin{tikzcd}
			X \arrow[d] \arrow[r] & \mathcal X \arrow[d] \\
			* \arrow[r]  & \Spec(A)            
		\end{tikzcd}$$
		such that $\mathcal X = \Spec \mathcal B$, where $(\mathcal B,\F)$ is a flat filtered $A$-algebra, and an isomorphism of filtered $\kk$-algebras $B \to \mathcal B \otimes A/\m$.
	\end{defi}
	
	
		The main result of this section is the description of the associated triple (the group of infinitesimal automorphisms, the tangent space, an obstruction space) to the filtered functor $\Def^{\geq 0}(B)$; this is the content of Theorem \ref{+obs}. Later sections heavily rely on this result.
	
	\begin{remark}
		Since $X$ is affine, we abuse the notation and denote $$\Def(B)=\Def(X),$$
		the same for variants with subscript.
		In general, we can define the same functor by saying that $X=\Spec_X(\O_X)$ with the structural sheaf being filtered of $\mathcal X$ by flat subsheaves over the base. The results in this section remain unchanged.
	\end{remark}
	There is a natural transformation $\Def^{\geq 0}(X) \to \Def(X)$ given by forgetting the filtration. 
	Now we prove that the filtered functor is indeed a good deformation functor. The idea of the proof is based on \cite[Proposition 3.1]{JJ-elementry}. We show an isomorphism of the filtered functor with an analogue of the multigraded Hilbert scheme.
	\begin{twr} \label{filt-good}
		Let $X=\Spec B$ be an affine $\G_m$-scheme. Then the functor
		$\Def^{\geq 0}(X)$ is a good deformation functor.
	\end{twr}
	\begin{proof}
		First, we build from $(B,F)$ a Rees algebra $\tilde B= \oplus_{i\in \Z} t^{-i}F_{i} \subset B[t,t^{-1}]$. It is a graded algebra flat over $\kk[t^{-1}]$. Note that the degree of $t^{-1}$ is $-1$, and $\kk[t^{-1}] \to \tilde B$ is a morphism of graded algebras. We consider the deformation functor
		$$ \Def^0(\tilde B/\kk[t^{-1}]),$$
		the functor of abstract equivariant deformations over $\kk[t^{-1}]$, see \cite[1.5]{Pinkham-Gm}
		\vskip 0.2 cm
		We define a natural transformation of functors of Artin rings:
		\begin{align}
			\Def^{\geq 0}(B)(A) &\to \Def^0(\tilde B/\kk[t^{-1}])(A), \\
			(\mcB, \F) &\mapsto \tilde \mcB =(\oplus_{i\in \Z} t^{-i}\F_{-i}).
		\end{align}
		where $A\in \Art_\kk$.
		It is well defined; the algebra $\tilde \mcB$ is graded and flat over the Artinian algebra $A[t^{-1}]$.
		Conversely we define 
		\begin{align}
			\Def^0(\tilde B/\kk[t^{-1}])(A)  &\to \Def^{\geq 0}(B)(A) ,\\
			(\tilde \mcB = \bigoplus \tilde \mcB_i) &\mapsto (\frac{\tilde \mcB }{(t-1)}, \frac{\tilde \mcB_i}{(t-1)})
		\end{align}
		We have shown an isomorphism of the functors, since the equivariant functor is a good deformation functor, the statement holds.
	\end{proof}
	
	\begin{notation}
		For the opposite filtration on $B$ in Definition \ref{defi-filt-fun}, we get $\Def^{\leq 0}(X)$.
	\end{notation}	
	
	\begin{twr} \label{+obs}
		Let $X$ an affine $\G_m$-scheme of finite type. The functor
		$\Def^{\geq 0}(X)$ is a good deformation functor with an associated triple
		$$(T^{0}_X)_{\geq 0}, \ (T^{1}_X)_{\geq 0}, \ (T^{2}_X)_{\geq 0}.$$
	\end{twr}
	As the most technical part of this section the proof of the Theorem \ref{+obs} is delegated to a separate Subsection \ref{proof+}. \vskip 0.2 cm
	\begin{stwr} \label{prorep-lem}
		If $(T^1_{B})_{\geq 0}$ is finitely dimensional, then $\Def^{\geq 0}(B)$ admits a semi-universal family.
		Moreover, if the group of non-negative infinitesimal automorphisms $(T^0_B)_{\geq 0}$ vanishes, then $\Def^{\geq 0}(B)$ is prorepresentable. 
	\end{stwr}
	\begin{proof}
		The first part of the statement is a direct consequence of Theorem \ref{+obs}, together with Theorem \ref{schless-twr}.
		For the second part, it is the classical result on prorepresentability \cite[Theorem 2.6.1]{sernesi_Deformations_of_algebraic_schemes} applied to the equivariant functor $\Def^0(\tilde B/\kk[t^{-1}])$, while keeping in mind that we are only concerned with $\G_m$-equivariant deformations.
	\end{proof}
	As $B$ is a graded algebra of finite type, there is an equivariant surjection $$S=\kk[{x_1},\ldots, x_n]\to B,$$ we consider the local Hilbert functor $\Hilb_{[B]}$, Definition~\ref{rmk-loc-hilb}. It is prorepresentable, when the "big" functor $\Hilb(\A^n)$ is representable. That is when the tangent space is finitely dimensional \cite[Corollary 3.2.2]{sernesi_Deformations_of_algebraic_schemes}.
	 We look at the functor of filtered deformations $\Hilb_{[B]}^{\geq 0}$,
	 it can be defined by taking the Białynicki-Birula decomposition of $\Hilb_d(\A^n)$ and take the deformation functor associated to completed local ring at the point associated to the algebra $B$.
	   This is also isomorphic to $\Def^{\geq 0}(X\subset \A^n)$, see Remark \ref{rmk-pairs-filt}.
	 \begin{lem} \label{smooth-nonstrict}
	 	There is a natural smooth morphism of good deformation functors ${\Hilb^{\geq 0}_{[B]} \to \Def^{\geq 0}(B)}$.
	 \end{lem}
	 \begin{proof}
	 	The map is given by forgetting the embedding of $X$ into $\A^n$.
	 	As this is a map between good deformation functors, by Theorem \ref{fantechi-manetti}, it suffices to check that we get a surjection on tangent spaces and an injection on obstructions. By the presentation of $T^1_{B}$-module \cite[Proposition~3.10]{hardef}, as a quotient of the tangent space to the Hilbert scheme, we get the surjection. 
	 	As an obstruction space for $\Hilb^{\geq 0}_{[B]}$ we can take $(T^2_B)_{\geq 0}$, see \cite[Theorem~4.2]{JJ-elementry}. Then the induced map of obstructions is an isomorphism.
	 	As a consequence, the map of filtered deformations is smooth as claimed.
	 \end{proof}
	For $B$ a graded algebra, let $\Def^0(B)$ denote the functor of $\G_m$-equivariant deformations of $B$. 

	\begin{lem} \label{cartesian-eq}
		Let $B$ be a graded $\kk$-algebra of finite type.
		The following diagram is Cartesian
		$$	\begin{tikzcd}
			\Def^0(B) \arrow[d] \arrow[r] & \Def^{\leq 0}(B) \arrow[d] \\
			\Def^{\geq 0}(B) \arrow[r]                 & \Def(B).            
		\end{tikzcd}$$
		
	\end{lem}
	\begin{proof}
		Let us fix an Artinian local algebra $A$, and consider the pullback of the bottom-left corner $\Def^{\geq 0}(B)\times_{\Def(B)}\Def^{\leq 0}(B)$.
		Let us take an $A$-point  $\mathcal B$ of this pullback 
		\begin{align}\label{seq}
			0\to I \to \mcB \to B \to 0.
		\end{align}
		It is equipped with two filtrations,
		$\F_i$ and $\mathcal G_{-i}$, such that when restricted to the closed point
		$$ \F_i\otimes \kk = F_i= \oplus_{j\geq i} B_j,$$
		$$ \mathcal G_{-i}\otimes \kk = \mathcal G_{-i}= \oplus_{j\leq i} B_j,$$
		see Definition \ref{defi-filt-fun}.
		Every element of filtration is flat over $A$.
		We prove that in such a case, the deformation $\mathcal B$ comes with a $\Z$-grading.
		The natural candidate for such a grading is:
		$$ \mathcal B_i = \F_i\cap \mathcal G_{-i}.$$
		The subalgebra $\oplus_{i} \mathcal B_i \subset \mathcal B$  is clearly graded. The question is whether the modules $\mathcal B_i$ are flat and whether the constructed algebra is the whole algebra $\mathcal B$.
		\vskip 0.2 cm
		Let us consider the following diagram given by tensoring the map given by subtraction ${\F_i \oplus \mcG_{-i} \to \mcB}$ with the short exact sequence \eqref{seq}
		$$\begin{tikzcd}
			0 \arrow[d]                                               & 0 \arrow[d]                                    &                  \\
			(\F_i\oplus \mathcal G_{-i})\otimes I \arrow[d] \arrow[r] & I \arrow[d] \arrow[r, two heads]  & \coker \arrow[d] \\
			\F_i\oplus \mathcal G_{-i} \arrow[r] \arrow[d, two heads] & \mcB \arrow[d, two heads] \arrow[r, two heads] & \coker \arrow[d] \\
			F_i\oplus G_{-i} \arrow[r, two heads]                     & B \arrow[r]                                    & 0.               
		\end{tikzcd}$$
		By the snake lemma, the right column is exact. By induction on small extensions, the top cokernel is trivial (in the case of a small extension, the bottom and top rows are isomorphic as $\kk$-vector spaces), hence the cokernel in the middle is also trivial.
		Since the module $\mathcal B_i$ fits in the following exact sequence
		$$ 0 \to \mathcal B_i \to \F_i\oplus \mathcal G_{-i} \to \mathcal B \to 0,$$
		it is flat.\\
		By a similar argument as above, but for $\oplus_i \mcB_i \to \mcB$, in the place of $\F_i \oplus \mathcal G_{-i}$, we get that $\oplus_i \mcB_i \to \mcB$ is an isomorphism of algebras, which finishes the proof.
	\end{proof}
	
	\begin{lem}[{\label{graded-structure-T1}\cite[Propositions 2.2-2.4]{Pinkham-Gm}}]
		Let $B$ be a graded $\kk$-algebra of finite type.
		Then the tangent space $T^1_B$ is graded, in particular $$T^1_B=(T^1_B)_{\geq 0}+(T^1_{B})_{\leq 0}= (T^1_{B})_{>0}\oplus (T^1_B)_0 \oplus (T^1_B)_{<0}.$$ \\
		Moreover, on the level of $T^2$, one has inclusions
		$$(T^2_B)_0\subset (T^2_B)_{\geq 0} \subset T^2_B.$$
	\end{lem}
	
	\begin{lem} \label{lem-cruc}
		Let $B$ be a graded $\kk$-algebra of finite type.
		Let us assume that $(T^1_B)_{<0}$ vanishes.
		Then the maps $\Def^0(B) \to \Def^{\leq 0 }(B)$  and $\Def^{\geq 0}(B) \to \Def(B)$ are smooth. 
	\end{lem}
	\begin{proof}
		For the first map, since the negative tangent space vanishes, by Lemma \ref{graded-structure-T1}, we get an isomorphism on tangent spaces and  an injection on obstructions.  
		By Theorem \ref{fantechi-manetti} we get smoothness. 
		
		\vskip 0.2 cm
		For the second map, consider the following pull-back diagram induced by Lemma \ref{cartesian-eq}.
		$$\begin{tikzcd}
			{(T^{1}_B)_0} \arrow[d, hook] \arrow[r, two heads] & {(T^{1}_B)_{\leq 0}} \arrow[d, hook] \\
			(T^{1}_B)_{\geq 0} \arrow[r, hook]                            & T^1_B.                      
		\end{tikzcd}$$
		By Lemma \ref{graded-structure-T1} we conclude that there is a surjection on the tangent spaces and an injection on obstructions.
		Henceforth, the conditions for Theorem \ref{fantechi-manetti} are satisfied, and we obtain smoothness.
	\end{proof}
	\begin{lem}\label{smoothly-eq} \label{abb-smooth}
		Let $B$ be a graded $\kk$-algebra of finite type.
		If $(T^1_B)_{<0}$ vanishes, $\Hilb^{\geq 0}_{[B]}$ are $\Hilb_{[B]}$ are smoothly equivalent.
	\end{lem}
	\begin{proof}
		By Lemma \ref{lem-cruc}, we know that $\Def^{\geq 0}(B) \to \Def(B)$ is smooth.
		Let us consider the following diagram
		$$\begin{tikzcd}
			{\Hilb^{\geq 0}_{[B]}} \arrow[d, "sm"] \arrow[r] & {\Hilb_{[B]}} \arrow[d, "sm"] \\
			\Def^{\geq 0}(B) \arrow[r, "sm"]                 & \Def(B).                      
		\end{tikzcd}$$
		This way, we get that both $\Hilb^{\geq 0}_{[B]}$ and $\Hilb_{[B]}$ are smooth over $\Def(B)$.
	\end{proof}
	
	\begin{remark}
		Note that, in the setting of Lemma \ref{smoothly-eq}, the functors $\Hilb^{\geq 0}_{[B]}$ and $\Hilb_{[B]}$ are smoothly equivalent, but the natural immersion $\Hilb^{\leq 0}_{[B]}\to \Hilb_{[B]}$ is not necessarily smooth. It can happen that the negative tangent space $(T^1_B)_{<0}$ vanishes, but there are non-trivial negative embedded deformations. In that case, the map $\Hilb^{\geq 0}_{[B]} \to \Hilb_{[B]}$ cannot induce surjection on the tangent spaces. In particular, the natural map is not smooth.
	\end{remark}
	
	\begin{defi} \label{abs-fibre}
		We define the functor of strictly negative (resp. positive) deformations $\Def^{<0}(X)$  as the following pullback in $\Art_{\kk}\to \Set$:
		$$\begin{tikzcd}
			\Def^{<0}(X) \arrow[d] \arrow[r] & \Def^{\leq0}(X) \arrow[d] \\
			* \arrow[r]                      & \Def^{0}(X).            
		\end{tikzcd}$$
	\end{defi}
	\begin{stwr} \label{fund-negative-def}
		The functor $\Def^{<0}(X)$ is a good deformation functor with the associated triple 
		$$(T^0_X)_{<0}, \ (T^1_X)_{<0}, \ (T^2_X)_{<-1}.$$
	\end{stwr}
	\begin{proof}
		Directly from the definition. We can choose the obstruction space to be $(T^2_X)_{<-1}$, instead of non-positive, since all the negative tangents have degree at most $-1$.
	\end{proof}
	A \emph{negative spike} of a closed $\G_m$-subscheme $X\subset \A^n$ is $\Hilb^{<0}_{[X]}$, the fibre of $\Hilb^{\leq 0}_{[X]} \to \Hilb^{0}_{[X]} $ over $[X]$. The next proposition says that the negative spike is smoothly equivalent to strictly negative deformations of the scheme.
	\begin{stwr} \label{smooth-strict}
		The natural transformation $\Hilb^{< 0}_{[B]} \to \Def^{< 0}(B)$ is smooth.
	\end{stwr}
	\begin{proof}
		This is a direct consequence of Corollary \ref{smooth-nonstrict}.
	\end{proof}
	
	\begin{remark}\label{rmk-pairs}
		Let $X \subset Y$ be a closed embedding of affine $\G_m$-schemes. Then we can repeat all the arguments in this subsection and get functors $\Def(X\subset Y)$, $\Def^{\leq 0}(X\subset Y)$, $\Def^{<0}(X\subset Y)$
		with the triples
		$$T^0_{X\to Y}, \ T^1_{X\to Y}, \ T^2_{X\to Y},$$
		$$(T^0_{X\to Y})_{\leq 0}, \ (T^1_{X\to Y})_{\leq 0}, \ (T^2_{X\to Y})_{\leq 0},$$
		$$(T^0_{X\to Y})_{<0}, \ (T^1_{X\to Y})_{<0}, \ (T^2_{X\to Y})_{<-1}.$$
		as well as fixed-target and fixed-source variants, following Definition \ref{defi-XY}, for a reference see \cite[Theorem 1.6]{Kleppe}.
		These functors admit smooth maps from the respective local Hilbert schemes.
	\end{remark}
	\begin{remark}\label{rmk-pairs-filt}
		In case of $X\subset Y$ a closed embedding of affine $\G_m$-schemes of finite type, all above applies except for the strictly negative deformations, since they are not defined.
	\end{remark}
	
	\begin{remark}
		Braid \eqref{braid0} is compatible with graded structures. As a consequence, we obtain a useful method to study the (non)-positive deformations of maps from Braid \eqref{braid0}.
	\end{remark}

	\begin{stwr} \label{prorep-lem-general}
		Let $X \to  Y$ be a closed embedding of $\G_m$-schemes.
		Assume that the negative tangent spaces $(T^1_X)_{<0}$, $(T^1_Y)_{<0}$ and $(T^1_{X/Y})_{<0}$ are finitely dimensional and that the the infinitesimal automorphisms $(T^0_X)_{<0}$, $(T^0_Y)_{<0}$ vanish.
		Then all three variants of deformations of the map $X\to Y$ 
		$$\Def^{< 0}(X \to Y),$$
		$$	\Def^{< 0}(X /Y),				$$
		$$	\Def^{< 0}(X \backslash Y)				$$
		are prorepresentable.
	\end{stwr}
	\begin{proof}
		Recall the following part of Braid \eqref{braid0}
		\begin{equation}	\label{braid-above} \tag{$\spadesuit_3$}
			\begin{tikzcd}
				0 \arrow[dd, bend right=49, Rightarrow ]  & & 0 \arrow[dd, bend left=49, dotted ]  \\
				& 0 \arrow[ld, dashed] \arrow[rd] & \\
				T^0_{X\backslash Y} \arrow[rd, Rightarrow] \arrow[dd, bend right=49, dashed ]                      &                                                           & T^0_{X / Y} \arrow[ld, dotted] \arrow[dd, bend left=49] \\
				& T^0_{X\to Y} \arrow[rd, Rightarrow] \arrow[ld, dotted]    &                                                                              \\
				T^0_Y                      &                                                           & T^0_X                                                        
			\end{tikzcd}
		\end{equation}
		By \cite[Proposition 2.3.1]{sernesi_Deformations_of_algebraic_schemes} the infinitesimal automorphisms of the local Hilbert functor $T^0_{X/Y}$ vanish.
		By specialising Braid \eqref{braid-above} above to the negative parts, simple diagram chasing shows that all negative $T^0$-modules vanish. \vskip 0.2 cm
		Arguing analogously on the $T^1$-level, we obtain that $(T^1_{X\backslash Y})_{<0}$, $(T^1_{X\to Y})_{<0}$ are finite.
		The rest of the argument is as in Corollary \ref{prorep-lem}.
	\end{proof}

	\begin{remark}
		Similar proof also works for the equivariant and the non-positive versions of the functors.
	\end{remark}
	
	\begin{stwr} \label{lem-tech-braid}
		Let $X\to Y$ be a morphism of $\G_m$-schemes. If the module $(T^2_{X\backslash Y})_{<0}$ vanishes, then the natural map $\Def^{<0}(X\to Y) \to \Def^{<0}(X)$ is smooth. Moreover, if $(T^2_{X\backslash Y})_{<0}$ vanishes, then it is isomorphic on the tangent spaces.
	\end{stwr}
	\begin{proof}
		From Braid \eqref{braid0}, we see that the image of $(T^1_{X\backslash Y})_{<0}$ in $(T^1_{X\to Y})_{<0}$ is the relative tangent, and $(T^2_{X \backslash Y})_{<0}$ is a relative obstruction space. We apply Theorem \ref{fantechi-manetti} to get the statement.
	\end{proof}

	\begin{stwr}  \label{step1}\label{lem-tech-braid-crucial}
		Let $X\to Y$ be a morphism of affine $\G_m$-schemes. If the module $(T^1_{Y})_{<-1}$ vanish then the composition of maps $\Def^{<0}(X/ Y) \to \Def^{<0}(X \to Y) \to \Def(X \to Y)$ is unobstructed. Moreover, if  $(T^1_Y)_{-1}$ vanishes, then the map is smooth.
	\end{stwr}
	\begin{proof}
		From Braid \eqref{braid0} we see that the image of $(T^1_{Y})_{<0}$ in $(T^1_{X/ Y})_{<0}$ is a relative obstruction space. This, by Theorem \ref{fantechi-manetti}, implies the second part of the statement. Since the obstructions lie in degrees~$<-1$, we get the first part of the statement.
	\end{proof}

\subsection{Proof of Theorem \ref{+obs} } \label{proof+}
This subsection is devoted to the proof of Theorem \ref{+obs}. As the proof is the most technical part of the argument; this section plays a role of an appendix to Section \ref{sec-3}.
We start with a technical lemma, which strengthens the isomorphism presented in the proof of Theorem \ref{filt-good}. Instead of comparing set-valued functors, we compare appropriate categories of modules. 
\begin{lem} \label{nat}
		Let $(B,F_i)$ be a filtered algebra, and $\tilde B= \bigoplus t^{-i}F_i$ is the graded Rees algebra associated with the filtration.
	The following natural transformation is an equivalence of categories:
	\begin{align*}
		\Mod^{filt}_{(B,F_i)} &\to \Mod^{gr}_{\tilde B},\\
		(M,\F_i) &\mapsto \bigoplus_{i\in \Z} t^{-i}\F_i.
	\end{align*}
\end{lem}
\begin{proof}
	We start by checking that this transformation is well-defined on morphisms. 
	We take a homomorphism $f\colon (M,\F)\to (N,\mathcal G)$ of filtered $B$-modules. We associate to it a map $$\tilde f \colon \bigoplus_{i\in \Z} t^{-i}\F_i \to \bigoplus_{i\in \Z} t^{-i}\mathcal G_i,$$ in the usual way. \vskip 0.2 cm
	Now we define the inverse transformation:
	\begin{align*}
		\Mod^{gr}_{\tilde B} &\to  \Mod^{filt}_{(B,F_i)},\\
		\tilde M &\mapsto \left(\frac{M}{(t-1)},\frac{M_i}{(t-1)}\right).
	\end{align*}
	In this manner, we get an inverse. Thereupon, the categories are isomorphic.
\end{proof}

\begin{proof}[Proof of Theorem \ref{+obs}] 
	By Theorem \ref{filt-good}, we have an isomorphism
		$$\Def^{\geq 0}(B) \cong  \Def^0(\tilde B/\kk[t^{-1}]).$$
	We construct the triple for the functor $\Def^0(\tilde B/\kk[t^{-1}])$.
	By \cite[Theorem 5.1]{hardef} the tangent space to $\Def(\tilde B/\kk[t^{-1}])$ is $T^1_{\kk[t^{-1}]}(\tilde B,\tilde B)$. By the argument of \cite[Theorem 10.1]{hardef} the functor admits an obstruction space $T^2_{\kk[t^{-1}]}(\tilde B,\tilde B)$.
	By \cite[Lemma 1.2.6]{sernesi_Deformations_of_algebraic_schemes} the group of infinitesimal automorphisms is $T^0_{\kk[t^{-1}]}(\tilde B,\tilde B)$.
	 By \cite[Theorem 2.3]{Kleppe} the obstruction theory of the equivariant functor is just the $0$-th graded piece of both the tangent and obstruction spaces. The same holds for infinitesimal automorphisms, as we take $\Aut_{\Spec B}$ functor and take the equivariant component of the tangent space.  
	 	This way we get that 
	 $$T^0_{\kk[t^{-1}]}(\tilde B,\tilde B)_0, \ T^1_{\kk[t^{-1}]}(\tilde B,\tilde B)_0, \ T^2_{\kk[t^{-1}]}(\tilde B,\tilde B)_0$$
	 is the triple for $\Def^{\geq 0}(B)$. \vskip 0.2 cm
	The last thing is to check is that 
	$T^i_{\kk[t^{-1}]}(\tilde B,\tilde B)_0 = (T^{i}_B)_{\geq 0},$
	for $i=0,1,2$.
	\vskip 0.2 cm
	Let us take $\tilde B$ and repeat \cite[Construction 3.1]{hardef} to construct the graded naive cotangent complex:
	$$	\begin{tikzcd}
		& 0                            &                             &                    &   \\
		0 \arrow[r] & \tilde I \arrow[r] \arrow[u] & \tilde S \arrow[r]          & \tilde B \arrow[r] & 0 \\
		& \tilde F \arrow[u]           &                             &                    &   \\
		& \tilde Q \arrow[u]           & \tilde F_0 \arrow[l] &                    &   \\
		& 0. \arrow[u]                  &                             &                    &  
	\end{tikzcd}$$
	Since $\tilde B$ is flat over $\kk[t^{-1}]$ after tensoring with $\kk[t^{-1}]/(t-1)$ we get an analogous diagram
	$$\begin{tikzcd}
		& 0                                  &                                   &                          &   \\
		0 \arrow[r] & \tilde I/(t-1) \arrow[r] \arrow[u] & \tilde S/(t-1) \arrow[r]          & \tilde B/(t-1) \arrow[r] & 0 \\
		& \tilde F/(t-1) \arrow[u]           &                                   &                          &   \\
		& \tilde Q/(t-1) \arrow[u]           & \tilde F_0/(t-1) \arrow[l, hook'] &                          &   \\
		& 0. \arrow[u]                        &                                   &                          &  
	\end{tikzcd}$$
	By the fact $\tilde B /(t-1) = B$ and Lemma \ref{nat}, we can treat this diagram as a diagram of filtered $S=\tilde S/(t-1)$-modules. Moreover, $S$ is a polynomial ring over $\kk$, as a consequence, this is a diagram for $B$ in \cite[Construction 3.1]{hardef}.
	From that, we see (going through Lemma \ref{nat}) that the naive cotangent complex for the filtered algebra $B$:
	$$L_\bullet= L_2 \to L_1 \to L_0$$
	is equivalent to the complex for the graded algebra $\tilde B$
	$$\tilde L_\bullet= \tilde{L_2} \to \tilde{L_1} \to \tilde{L_0}.$$
	Therefore as $T^i_B=T^i_{\kk}(B,B)=h^i(\Hom(L,B))$ we get the following equalities $$(T^{i}_B)_{\geq 0}=H^i(\Hom(L_\bullet,B))_{\geq 0}= H^i(\Hom^{filt}(L_\bullet,B))= H^i(\Hom(\tilde L_\bullet,\tilde B)_0)=$$ $$=H^i(\Hom(\tilde L_\bullet,\tilde B))_0=T^i_{\kk[t^{-1}]}(\tilde B,\tilde B)_0.$$
\end{proof}

\section{Pointed deformations}\label{sec-pointed}

	\begin{defi}\label{pointedGmscheme-defi}
	A morphism of affine  $\G_m$-schemes $*\to Y$ is \emph{ a pointed affine $\G_m$-scheme}.
	\end{defi}
	 If $Y$ is a pointed $\G_m$-scheme, then by definition the distinguished point is $\G_m$-fixed.
	It is worth mentioning the following
 	$$(T^1_Y)_{<-1}=0$$
	condition plays a remarkable role in all the results in this section. The reason for that is Corollary~\ref{lem-tech-braid-crucial} applied for $(X\to Y)= (*\to Y)$.
\begin{defi}
	We say that a $\G_m$-scheme $Y$ has trivial negative tangents (denoted as TNT) if $(T^1_Y)_{<0}=0$.
\end{defi}

\begin{lem}\label{opent2}
	Let $d$ be an integer.
	The conditions that $T^2(\O_Y,\kk)_{<-d}=0$, $(T^1_Y)_{<d}=0$, $(T^0_Y)_{<d}=0$ and the condition that $Y$ is TNT are open in flat $\G_m$-equivariant families.
\end{lem}
\begin{proof}
	Fix $i\in \{0,1,2\}$.
	Let us take a flat family $\mathcal Y \to S$ with the special fibre $Y$. Since $T^i_{\Y/S}$ is a coherent sheaf, its rank is upper-semicontinuous. We can endow $S$ with a trivial $\G_m$-action, then we see that the rank of any graded part is also upper-semicontinuous.
\end{proof}

\subsection{Prorepresentability}
This subsection has an auxiliary character and is focused on the proreprehensibility of various good deformation functors. In Proposition \ref{pt-in-Y} we show that the Hilbert scheme of a point $\Hilb_1(Y)$ is naturally isomorphic to the negative spike of the scheme $Y$. 
Then we comment on prorepresentability of negative deformations of finite schemes in Lemma~\ref{t0-fix}. The main part of the section is a comparison of the Hilbert scheme of a point on a finite scheme $Y$ and the functor of negative pointed deformations of $Y$ in Corollary \ref{iso-pointed}. At the end we check the necessity of the assumptions in Lemma \ref{step1-tan}.
\vskip 0.2 cm

 Let us describe the functor of embedded deformations of a distinguished point in a scheme. This functor is a source of all "fractal" results known to the author.
\begin{prop}\label{pt-in-Y}
	The functor $\Def(*/Y)$ is isomorphic to $\Def^{<0}(*/Y)$ and both are prorepresented by the completed local ring at the distinguished point of $Y$. The tangent spaces of those functors are concentrated in $(-1)$ times the degrees of the minimal generators of the maximal ideal of $\O_Y$. 

\end{prop}
\begin{proof} 
	Let us take the exact sequence describing the pointed structure
	$$ 0 \to \m \to \O_Y \to \kk.$$
	The tangent space of the functor is the tangent to the Hilbert scheme at the same point
	$$ \Hom_{\O_Y}(\m,\kk).$$
	Hence, we get the last part of the statement, since $\m$ is generated in positive degrees, the tangent space is strictly negative.
	By Theorem \ref{fantechi-manetti} we get that $\Def^{<0}(*/Y)\to \Def(*/Y)$ is smooth, as the map of obstruction spaces is the top horizontal map in the commutative diagram
	$$\begin{tikzcd}
		(T^2_{*/Y})_{<0} \arrow[d, hook] \arrow[r] & T^2_{*/Y} \arrow[d, hook] \\
		{\Ext^{1}_{\O_Y}(\m,\kk) }_{<0} \arrow[r]      & {\Ext^1_{\O_Y}(\m,\kk)}.    
	\end{tikzcd}$$
	Therefore, the first part of the statement holds, as we get a smooth map ${\Def^{<0}(*/Y)\to \Def(*/Y)}$, between prorepresentable functors, which is isomorphic on the tangent spaces.
	The Hilbert scheme of one point on $Y$ is isomorphic to $Y$. Consequently, the local functor is as claimed in the second part of the statement.
\end{proof}
	\begin{stwr} \label{pre-fractal}
		Assume that $(T^1_Y)_{<-1}$ vanishes and that $\O_Y$ is generated in degree $1$. Then $(T^1_{*\to Y})_{<-1}$ also vanishes. Moreover, the following diagram is Cartesian up to smooth equivalence
		$$\begin{tikzcd}
			\Def(*/Y) \arrow[d] \arrow[r] & \Def(*\to Y) \arrow[d] \\
			* \arrow[r]                   & \Def(Y).               
		\end{tikzcd}$$
	\end{stwr}
	\begin{proof}
		The first part is a direct consequence of the long exact sequence of $T^i$-functors, see Braid~\eqref{braid0}. 
		$$ T^1_{*/Y} \to T^1_{*\to Y} \to T^1_Y.$$
		By Lemma~\ref{lem-tech-braid-crucial}, we get that the map $\Def(*/Y) \to \Def(*\to Y)$ is unobstructed, thus by Lemma~\ref{technical-deformation} we get that the diagram is Cartesian up to smooth equivalence.
	\end{proof} 
	Since we are predominantly interested in finite local schemes, one could wonder if, in that case, the condition for prorepresentability in Proposition \ref{prorep-lem}, $(T^0_Y)_{<0}=0$ is always satisfied.
In characteristic $0$, if the structural sheaf $\O_Y$ is generated in a fixed positive degree, it is true. 
\begin{lem} \label{t0-fix}
	Assume that $Y$ is finite and the structural sheaf $\O_Y$ is generated in a fixed positive degree.
	In characteristic zero or larger than the length of $Y$, the group of infinitesimal automorphisms $(T^0_Y)_{<0}$ vanishes. And thus $\Def^{<0}(Y)$ is prorepresentable.
\end{lem}
\begin{proof}
	An element of $T^0_Y$ is a derivation $\Der_\kk(\O_Y,\O_Y)$. If such a derivation $\delta$ is negative and non-zero, then $\O_Y$ admits an element $e$ sent to an invertible element.
	For the minimal positive integer $k$ such that $e^k=0$ we get
	$$ 0=\delta(e^k)=k!e^{k-1}\delta(e),$$
	and since $\chr \kk >k$ we get a contradiction.
	Since the negative tangent space is finitely dimensional, the last part is an application of {Corollary~\ref{prorep-lem}}.
\end{proof}

\begin{stwr} \label{iso-pointed}
	Assume that $Y$ admits $\G_m$-action.
	If $(T^0_Y)_{<0}$ and $(T^1_Y)_{<0}$ vanish, then $$\Def(*/Y)\to\Def^{<0}(*\to Y)$$
	is an isomorphism.
\end{stwr}
\begin{proof}
	By Corollary \ref{step1} we see that the map is smooth.
	The kernel of the map $T^1_{*/Y} \to T^1_{*\to Y}$ is the image of $T^0_Y$. By our assumption, the injectivity of the tangent map follows.
	We have a smooth map inducing an isomorphism on tangent spaces, and $\Def^{<0}(*/Y)$ is prorepresentable. It is enough to prove that the target functor $\Def^{<0}(*\to Y)$ is prorepresentable. We argue that this holds by Corollary \ref{prorep-lem-general}. The assumption that the module $(T^1_{*/Y})_{<0}$ is finite is satisfied, since $Y$ has a finitely dimensional tangent space at the distinguished point.
\end{proof}
In the following lemma, we show that the condition on the infinitesimal deformations is necessary for Corollary \ref{iso-pointed} above.

\begin{lem} \label{step1-tan}
		Let $Y$ be an affine $\G_m$-scheme, with the structural sheaf $\O_Y$ generated in a fixed positive degree. Then the map $(T^0_Y)_{<0} \to (T^1_{*/Y})_{<0}$ is injective.
		If $(T^0_Y)_{<0}$ does not vanish, the map of the tangent spaces
		$$ T^1_{*/Y} \to T^1_{(*\to Y)}$$
		is not bijective.
\end{lem}
\begin{proof}
	
			We look at the exact sequence of $T^i$ functors induced by the following exact sequence of $\O_Y$-modules
	$$ 0 \to \m \to \O_Y \to \kk \to 0,$$
		\begin{equation}\label{shor-t0}
	 0 \to T^0_\kk(\O_Y,\m) \to T^0_\kk(\O_Y,\O_Y) \to T^0_\kk(\O_Y,\kk)=T^1_{*\backslash * \to Y/Y}.			
 	\end{equation}
			This is part of Braid \eqref{braid0} for the pointed map $(*\to Y)$ given by the single arrow.
	\begin{equation} \label{bradi-t0} \tag{$\spadesuit_4$}
	\begin{tikzcd}[sep=small]
	T^0_{*\backslash Y} \arrow[rd, dotted] \arrow[dd, bend right=49]                      &                                                           & T^0_{* / Y} \arrow[ld, Rightarrow] \arrow[dd, dashed, bend left=49] \\
	& T^0_{*\to Y} \arrow[rd, dotted] \arrow[ld, Rightarrow]    &                                                                              \\
	T^0_Y  \arrow[dd, Rightarrow, bend right=49] \arrow[rd]                       &                                                           & 0 \arrow[dd, dotted, bend left=49] \arrow[ld, dashed]                    \\
	& T^1_{*\backslash * \to Y/Y} \arrow[ld, dashed] \arrow[rd] &                                                                              \\
	T^1_{*/ Y} \arrow[rd, Rightarrow]  &                                                           & T^1_{*\backslash Y} \arrow[ld, dotted]                       \\
	& T^1_{*\to Y}      &                                                                              
	\end{tikzcd}	\end{equation}
			By the definition \cite[Proposition 3.6]{hardef}, the sequence \eqref{shor-t0} is equal to
			$$ 0 \to \Der_\kk(\O_Y,\m) \to \Der_\kk(\O_Y,\O_Y) \to \Der_\kk(\O_Y,\kk).$$
			Since, the module $\Der_\kk(\O_Y,\m)$  is zero in negative degrees, the map $$(T^0_Y)_{<0} = T^0_\kk(\O_Y,\O_Y)_{<0} \to (T^1_{*\backslash * \to Y/Y})_{<0}$$ is injective.
			Hence, when composed with the map $T^1_{*\backslash * \to Y/Y} \to T^1_{*/Y}$, which is injective by Braid~\eqref{bradi-t0}, we get the statement.
\end{proof}

\subsection{Fractal family on the pointed Hilbert scheme}
In this subsection we define a notion of a fractal family, inspired by \cite[Section 3.3]{szachniewicz-nonreduce}. Then we show the existence of such a family on the pointed Hilbert scheme $\Hilb_{1,d}$, this is the main result of this subsection Theorem \ref{fractal1}, which is a generalisation of \cite[Theorem 2.8]{Unexpected-GGGL}.

\vskip 0.2 cm

We start with an example, which showcases the difference between functor of pointed the usual deformations of a local scheme.
\begin{exm}
	Let $\chr \kk =p $ and take a finite scheme $Y=\Spec(\kk[x]/x^q)$. Then let us consider the unique pointed structure on $Y$ given by:
	$$ 0 \to (x) \to \O_Y \to \kk \to 0.$$
	The functor $\Def(*/ Y)$ is the Hilbert scheme of a point on $Y$ it is isomorphic to $Y$ itself. The tangent space at the unique closed point is
	$$ \Hom_{\O_Y}((x),\kk),$$
	it is one-dimensional and concentrated in degree $-1$. Therefore, all those embedded deformations are negative.
	The image of the maximal deformation $\Def(*/Y)(\O_Y)$ in $\Def(*\to Y)$ is
	$$\kk[\eps]/(\eps^q) \to \kk[x,\eps]/((x+\eps)^q,\eps^q) \xrightarrow{x\mapsto 0} \kk[\eps]/(\eps^q).$$ 
	if $p|q$ then this is a trivial pointed deformation, since $(x+\eps)^q=x^q+\eps^q$. 
	But in case $p\nmid q$ we show below that this pointed deformation is not trivial. \vskip 0.2 cm
 	The argument is by a contradiction, we assume that there is an isomorphism of the pointed deformation described above and the trivial one.
	We have a pair of pointed deformations, over $\kk[\eps]/(\eps^q)$, since they are pointed, each of them can be seen as a deformation with a section.
	$$\begin{tikzcd}
		{\kk[x,\eps]/(\eps^q,(x+\eps)^q))} \arrow[rr, "\cong"] \arrow[rdd, "\pi_1", bend left] &                                                                       & {\kk[x,\eps]/(\eps^q,x^q)} \arrow[ldd, "\pi_2", bend left] \\
		&                                                                       &                                                            \\
		& {{\kk[\eps]/(\eps^q)}.} \arrow[ruu, bend left] \arrow[luu, bend left] &                                                           
	\end{tikzcd}$$
	Let us denote the claimed isomorphism by $\phi\colon{\kk[x,\eps]/(\eps^q,(x+\eps)^q))} \to {\kk[x,\eps]/(\eps^q,x^q)} $.
	First, by the commutativity of the diagram above
	$$\phi(\eps)=\eps.$$
	Second, since $\phi$ lies over identity on $\kk[x]/(x^q)$
	$$ \phi(x) = x  \mod(\eps),$$
	equivalently
	$$ \phi(x)-x \in (\eps) \subset \kk[x,\eps]/(\eps^q,x^q).$$
	Hence, there exists $h\in \kk[x,\eps]/(\eps^q,x^q)$ such that 
	$$ \eps h = \phi(x)-x,$$
	but since 
	$$\eps\pi_2(h)=\pi_2(\eps h)=\pi_2(\phi(x)-x)=\pi_1(x)-\pi_2(x)=0-0=0,$$
	therefore $h\in (\eps^{q-1},x)\subset \kk[x,\eps]/(x^q,\eps^q)$.
	This implies that are elements $g,g'\in \kk[x,\eps]/(x^q,\eps^q)$ such that $$ h=xg+\eps^{q-1}g'.$$
	Thus $\phi(x)=x(1+g\eps)$ therefore $\phi(x)$ is nilpotent of order $q$,
	but $x^q \in \kk[x,\eps]/(\eps^q,(x+\eps)^q)$ is not zero, a contradiction with $\phi$ being an isomorphism.
\end{exm}
The example above showcases a phenomenon first shown in \cite{Unexpected-GGGL}. We took a finite pointed scheme and translated it affinely, while not \textit{moving} the distinguished point. This provided us with nontrivial deformations, all first-order deformations used in \cite[Theorem 2.8]{Unexpected-GGGL} are constructed this way. In that paper, the authors deduced from that construction the non-reducedness of the pointed Hilbert scheme. We focus on the singularity type, which is a more refined invariant.

\begin{defi}\label{fractal}
	Let $\U\to \M$ be a flat morphism with irreducible fibres and with a section $\M\to \U$. Let $y\in \M$ be $\kk$-point such that $\U_{|\{y\}}= \U\times_{\M} \{y\}$ is a finite non-reduced scheme.
	We say that $\M$ admits \emph{a fractal structure} at a point $y$ if there is a morphism $\{y\}\to Z\to \M$ such that the pointed schemes $(\U_{|Z},\{y\})$ and $(\M,\{y\})$ are smoothly equivalent.
	$$\begin{tikzcd}
		&                                                 & \M                       \\
		\U_{|\{y\}} \arrow[r] \arrow[d]      & \U_{|Z} \arrow[d] \arrow[r] \arrow[ru, "sm-eq"] & \U \arrow[d]             \\
		\{y\} \arrow[r] \arrow[u, bend left] & Z \arrow[r] \arrow[u, bend left]                & \M. \arrow[u, bend left]
	\end{tikzcd}$$
\end{defi}

We apply the definition for $\M$ being a moduli space of pointed objects, for example, the pointed Hilbert scheme of points $\Hilb_{1,m}$.
The definition above says that the completed local ring of the moduli space at $\{y\}$ looks like a deformation of $\U_{|\{y\}}$ over a scheme $Z$.
	\begin{lem}\label{frac-nonred}
		Let $\U\to \M$ be a fractal family at a $\kk$-point $y \in \M$. Then the $\kk$-point $y$ is a non-reduced point of $\M$.
	\end{lem}
	\begin{proof}
		Let us consider the following diagram.
		$$\begin{tikzcd}
			\U_{|\{y\}} \arrow[r] \arrow[d]      & \U_{|Z} \arrow[d]      \\
			\{y\} \arrow[r] \arrow[u, bend left] & Z \arrow[u, bend left]
		\end{tikzcd}$$ 
		Since the morphism $\U \to \M$  satisfies the conditions of Lemma~\ref{lemma-section} and since the fibre $\U_{|\{y\}}$ is not a point, $\U_{Z}$ is non-reduced.
		By Definition \ref{fractal}, $\M$ is also non-reduced at $y$.
	\end{proof}

 To get any kind of \textit{fractality} we need to assure that our pointed local scheme $Y$ does not deform to a reducible one.
\begin{defi}[{\cite[Definition 5.15]{Bucz-Klepp-cleav}}]
	We say that a finite local scheme $Y$ is \emph{uncleavable} if for every flat finite family over irreducible base $Z$
	$$\begin{tikzcd} 
		Y \arrow[d] \arrow[r] & \Y \arrow[d,"\pi"] \\
		* \arrow[r]           & Z           
	\end{tikzcd}$$
	with $Y$ as the fibre over the special point, the general fibre of $\pi$ is local.
\end{defi}

\begin{defi}
	A component of the Hilbert scheme $\Hilb_m(\A^n)$ is called \emph{elementary} if it parametrises subschemes supported at a single point. 
\end{defi}

\begin{lem}\label{michal-section}
	Let $V \subset \Hilb_m{\A^n}$ be an elementary component. Then the universal family $\pi \colon \U_V \to V$ admits a section $\sigma$.
\end{lem}
	\begin{proof}
		By \cite[Proposition 2.27]{szachniewicz-nonreduce}, the component $V\cong V_0\times \A^n$, where $V_0$ parametrizes schemes supported at the origin $0\in \A^n$.
		This way we also obtain a decomposition of the universal family $\U_V = \A^n \times \U_{V_0}$.
		Hence it is enough to construct a section of the restricted projection ${\pi_0 \colon \U_{V_0} \to V_0}$.
		By the definition of the universal family $\U_{V_0}$ is a closed subscheme of $\A^n \times V_0$, since $V_0$ parametrises schemes supported at the origin, $\U_{V_0}$ contains the subscheme $\{0\}\times V_0 \cong V_0$.
			\end{proof}

We show that the nested Hilbert scheme is a source of fractal structures.

\begin{twr} \label{fractal1}
	Let $U$ be a smooth quasi-projective variety and let $m\geq 2$ be a positive integer. Then, the nested Hilbert scheme $\Hilb_{1,m}(U)$ admits a fractal structure at any point ${[*,X]\in \Hilb_{1,m}(U)}$, such that $[X]\in \Hilb_m(U)$ and $X$ is uncleavable.
\end{twr}
The theorem is a generalisation of \cite[Theorem 2.8]{Unexpected-GGGL}.
\begin{remark}
	Since the property of being uncleavable is weaker than lying on unique elementary component, we get a generalisation of \cite[Theorem 2.8]{Unexpected-GGGL} in two ways.
First, we consider a larger locus of points. Second, not only do we show non-reducedness, but also describe structure of the local rings.
\end{remark}
\begin{proof}
	Let us fix $X\subset U$ to be uncleavable finite scheme of length $m$.
	
	Let $V$ be an open neighbourhood of $[X]\in \Hilb_m(U)$, such that $V$ parametrises only cleavable points.
	By Lemma \ref{michal-section} the universal family $\pi \colon \U_V \to V$ admits a section $\sigma_0$, hence we get the induced map to the nested Hilbert scheme $\Hilb_{1,m}(U)$.
	$$\begin{tikzcd}
		\U_V \arrow[d, "\pi"', bend right] \arrow[r] & \U \arrow[d, bend right]            \\
		V \arrow[r] \arrow[u, "\sigma"', bend right] & {\Hilb_{1,m}(U)} \arrow[u, bend right]
	\end{tikzcd}$$
	
	By definiton, a $S$-point of the universal family $\U_V$ is a pair consisting of a closed subscheme $Y_S$ of $S\times U$, flat over $S$, and a $S$-point of $Y_S$, hence a section of $Y_S\to S$. This way we realise $\U_V$ as a subscheme of $\Hilb_{1,m}(U)$.
	Let $\widetilde V\subset \Hilb_{1,m}(U)$ be the pullback as follows
	$$\begin{tikzcd}
		\widetilde V \arrow[d, "\pi_2"] \arrow[r] & {\Hilb_{1,m}(U)} \arrow[d] \\
		V \arrow[r]                               & \Hilb_m(U)               
	\end{tikzcd}$$
	Note that the morphism $\sigma$ is also a section of projection $\pi_2$.
	Let us choose a $\kk$-point $[Y]$ in $V$ and consider the Cartesian diagram
		$$\begin{tikzcd}
			W \arrow[d] \arrow[r] & \widetilde V \arrow[d, "\pi_2"] \\
			* \arrow[r]           & V.                 
		\end{tikzcd}$$
		 We transfer from embedded deformations to the abstract ones, and look at the following diagram, see {Remark~\ref{rmk-pairs}}
		\begin{equation}\label{pre-frac-diag}
		\begin{tikzcd}
			\Def(*/Y) \arrow[d] \arrow[r] & \Def(*\to Y) \arrow[d] \\
			* \arrow[r]                   & \Def(Y).               
		\end{tikzcd}	
	\end{equation}
		The section $\sigma$ induced a section of $\Def(*\to Y) \to \Def(Y)$.
			Hence the exact sequence 
		$$ T^1_{*/Y} \to T^1_{*\to Y} \to T^1_Y$$
		splits, in particular $T^1_Y\to T^2_{*/Y}$ is zero, and by Braid \eqref{braid0}, we get that the map ${\Def(*/Y) \to \Def(*\to Y)}$ is unobstructed, thus by Lemma~\ref{technical-deformation} we get that Diagram \eqref{pre-frac-diag} is Cartesian.
		If so, the induced map ${W\to\Def(*/Y)\cong Y}$ is smooth. Since the map is isomorphic on the tangent spaces, and the target functor is prorepresentable, the map is an isomorphism. Hence $W\cong Y$. This way get that the morphism of $V$-schemes
		$$\begin{tikzcd}
			\U_V \arrow[rd] \arrow[rr] &   & \widetilde V \arrow[ld] \\
			& V &              
		\end{tikzcd}$$
		inducing isomorphisms on the fibres. Since $\pi \colon \U_V \to V$ is flat, so is $\pi_2\colon \widetilde V \to V$, by \cite[\href{https://stacks.math.columbia.edu/tag/039D}{Tag 039D}]{stacks-project}.
		Hence we get a surjective homomorphism of locally free $\O_V$-modules of rank $m$
		$$ (\pi_2)_* \O_{\widetilde{V}} \to \pi_* \O_{\O_V}$$
		by Nakayama's lemma this is an isomorphism.
\end{proof}

\begin{stwr}\label{pointed-nonreduced}
		Let $Y$ be a finite uncleavable scheme. The functor $\Def(*\to Y)$ (resp. ${\Def^{<0}(*\to Y)}$) is non-reduced.
	\end{stwr}
	\begin{proof}
		The scheme $Y$ can be embedded into $\A^n$ for some large enough $n$. The map ${\Hilb_{[Y]}(\A^n) \to \Def(Y)}$ is smooth. Since $Y$ is uncleavable, $[Y]\in \Hilb_{m}(\A^n)$ satisfies the conditions of Theorem \ref{fractal1}, and we get a fractal family at the point $[Y]$. By Lemma \ref{frac-nonred} we get the non-reducedness. \vskip 0.2 cm
		For the second variant of the statement, we repeat the argument for Theorem \ref{fractal1}, for the negative deformations instead of the general ones. Recall that by Proposition \ref{pt-in-Y} we get $\Def^{< 0}(*/Y) \cong \Def(*/Y)$.
	\end{proof}

\subsection{Deformation of the colength $1$ subalgebra of a Gorenstein algebra}\hfill\\

In this subsection, we show a second example of fractal structure, recall Definition \ref{fractal}. We show that the nested Hilbert scheme $\Hilb_{m-1,m}(\A^n)$ admits a fractal structure. Moreover we reprove \cite[Theorem 1.5]{szachniewicz-nonreduce}, which describes the \textit{fractal-like} structure of the Hilbert scheme.
\vskip 0.2 cm
 For this subsection, we specialise the setting of the last subsection to the Gorenstein case.
\begin{set}
We take a scheme $Y=\Spec(B)\subset \A^n$ of length $m$, where $(B,\m)$ is a graded local Artinian Gorenstein algebra supported at the origin, and such that ${(T^1_B)_{<-1}=0}$. Let $g\in B$ be a homogeneous generator of the socle. Denote the degree of $g$ as $d$.
\end{set}
Recall that the deformation functor $\Def(*/Y)$ is prorepresented by $B$, see Lemma \ref{pt-in-Y}. We think about it as the source of fractalness, henceforth we fit it in the context of Gorenstein schemes.
\begin{lem}\label{gor-simply}
	There is an isomorphism of the good deformation functors
	$$ \Def(\Spec(\frac{B}{(g)})/Y)  \cong \Def(*/ Y).$$
\end{lem}
\begin{proof}
 	Let us fix $(A,\mathfrak n) \in \Art_\kk$. 
 	\begin{equation} \label{eq-gor1}
 		0 \to I \to B \otimes A \to A  \to 0 \in \Def(*/Y)(A).
 	\end{equation}
 	We construct an element in $\Def(\Spec(B/(g))/ Y)(A)$. First, we dualise Sequence \eqref{eq-gor1} and we get
 	\begin{equation*} \label{eq-gor2}
 		0 \to  A^{\vee}  \to (B \otimes A)^\vee \to I^\vee \to 0.
 	\end{equation*}
 	Since $B\otimes A$ is Gorenstein $(B\otimes A)^{\vee} \cong (B\otimes A) [d]$. We shift the degree to obtain
 	\begin{equation} \label{eq-gor3}
 		0 \to  A[-d]  \to B \otimes A \to I^\vee[-d] \to 0,
 	\end{equation}
 	this way $A[-d]$ is an ideal generated in degree $d$, being a free rank one $A$-module.
 	After tensoring with $A/\mathfrak n$, we get
 	\begin{equation} \label{eq-gor4}
 		0 \to  \kk[-d]  \to B \to \coker \to 0,
 	\end{equation}
 	since $(g)$ is the socle ideal and $\deg(g)=d$, Sequence \eqref {eq-gor4} is the unique element in $\Def(\Spec(B/(g))/ Y)(\kk)$. Hence, Sequence \eqref{eq-gor3} is an element in $\Def(\Spec(B/(g))/ Y)(A)$. Since the order of the argument is reversible, we get an isomorphism.
\end{proof}

Now we study the problem of deforming the closed embedding $\Spec(B/(g))\to Y$. 
\begin{notation}\label{defi-Yempty}
	The closed subscheme $\Spec(B/(g))\to Y$ is denoted by $Y\#\emptyset$. 
\end{notation}
The notation above is to be compared with definition of connected sum, Definition \ref{conn-defi}. Note that $Y\#\emptyset$ in general is not Gorenstein.
\begin{lem}\label{fractal2}
	The nested Hilbert scheme $\Hilb_{m-1,m}(\A^n)$ admits a fractal structure at point $[Y\#0,Y]$ such that $Y$ is uncleavable.
\end{lem}
\begin{proof}
	The argument is the same as in Theorem \ref{fractal1}, but with the following diagram
	 $$\begin{tikzcd}
		\Def(Y\#\emptyset/Y) \arrow[d] \arrow[r] & \Def(Y\#\emptyset\to Y) \arrow[d] \\
		* \arrow[r]                   & \Def(Y).               
	\end{tikzcd}$$
\end{proof}

The next step is to show that the forgetful functor
$$ \Def^{<0}(Y\#\emptyset \to Y) \to \Def^{<0}(Y\#\emptyset),$$
with appropriate assumptions is smooth.

\begin{lem}\label{lema-gor1}
	 Assume that the functor $T^2(\O_Y,\kk)_{<-d}$ vanishes.
	Then the forgetful map
	$$ \Def^{<0}(Y\#\emptyset \to Y) \to \Def^{<0}(Y\#\emptyset)$$
	 is smooth and isomorphic on the tangent spaces.
	 Moreover, if $\O_Y$ is generated in degree $1$, then $(T^1_{Y\#\emptyset})_{<-1}$ vanishes.
\end{lem}
\begin{proof}
	By Corollary \ref{lem-tech-braid} it is enough to check the vanishing of $$(T^i_{Y\#\emptyset \backslash Y})_{<0}=0,$$
	for $i=1,2$. \vskip 0.2 cm
	$$T^i_{Y\#\emptyset \backslash Y}= T^i_\kk(\O_Y,(g))= T^i_\kk(\O_Y,\kk)[\deg(g)].$$
	For $i=1$, the vanishing in negative degrees follows from the fact that $\deg(g)=d$ is the regularity of $Y$. For $i=2$, this is the assumption.\\
	The second part of the statement follows from the fact that $(T^1_{Y\#\emptyset \to Y})_{<-1}$ vanishes. To prove that, we take a part of the long exact sequence of $T^i$ functors
	$$ T^1_{Y\#\emptyset/Y} \to T^1_{Y\#\emptyset \to Y} \to T^1_{Y}.$$
	The module $T^1_{Y\#\emptyset/Y}$ is concentrated in degree $-1$ as by Lemma \ref{gor-simply} it is isomorphic to $T^1_{*/Y}$, which in turn by Lemma \ref{pt-in-Y} is concentrated in degree $-1$, as $\O_Y$ is generated in degree $1$.
	The vanishing of $(T^1_{Y})_{<-1}$ follows by the assumptions.
\end{proof}
That gives us the fractalness from \cite[Theorem 3.27]{szachniewicz-nonreduce}.
\begin{defi}[{\cite[Definition 3.2]{szachniewicz-nonreduce}}] \label{cubic-very}
	Let $F \in P$, be a cubic in $n$ variables. We call it very general if the following conditions are satisfied:
	\begin{itemize}
		\item $\Apolar(F)$ has Hilbert function equal to $(1,n,n,1)$.
		\item The minimal graded free resolution of $\Apolar(F)$ has generators in degree $2$ and syzygies in degree $3$.
		\item $\Apolar(F)$ admits TNT property.
	\end{itemize}
\end{defi}

\begin{lem}[{\cite[Proposition 2.16.]{Robert-Gorenstein}\label{cubic-positive}}]
		Let $F$ be a cubic, such that $\Apolar(F)$ has the Hilbert function $(1,n,n,1)$, then the positive tangent space $(T^1_{\Apolar(F)})_{>0}$ vanishes.
\end{lem}
\begin{proof}
	From \cite[Proposition 2.16.]{Robert-Gorenstein} we know that the tangent space in the embedded case vanishes in degrees $>1$.
	We take the presentation coming from Definition \ref{def-apolar}
	$$0\to I \to S\to \Apolar(F) \to 0.$$
	The embedding above presents $\Apolar(F)$ as a point on the Hilbert scheme.
	The positive tangent space to the Hilbert scheme at $\Apolar(F)$ is concentrated in degree $1$ and of dimension ${n+1 \choose 2}-n$.
	Let us denote $\Apolar(F)$ by $B$.
	By \cite[Theorem 3.10]{hardef} we can compute the dimension of $T^1$ from the following sequence:
	$$0\to  \Der(B,B) \to \Der(S,B) \to \Hom_B(I/I^2,B) \to T^1(B,B) \to 0 .$$
	Since the degree $1$ part $\Der(B,B)_{1}$ has dimension $n^2-{n+1 \choose 2}-n$ and $\Der(S,B)_1$ has dimension $n^2$. This concludes the statement.
\end{proof}
 The local Gorenstein algebras with Hilbert $(1,n,n,1)$ have been extensively studied, for example it is known that the positive spike over such an algebra is smooth, see \cite[Proposition 1.1]{JJ-class-gor}, \cite[Theorem 3.3]{EliasRossi12}. Furthermore, some properties of their moduli space are known.
 Recall, Definition \ref{HGor-defi}, that $\HG{m}{n}$ denote the Gorenstein locus of the Hilbert scheme of $m$ points on $\A^n$.
 \begin{prop} \label{cubic-component}
 	Let $n$ be a positive integer.
   	The locus of algebras with the Hilbert function $(1,n,n,1)$ is open and dense in a generically smooth irreducible component of $\HG{2n+2}{n}$.
 \end{prop}
 \begin{proof}
 	Let $n\leq 5$, by \cite[Theorem A]{CGJ}, the Gorenstein locus is irreducible, thus contained in the smoothable component, which is generically smooth.\vskip 0.2 cm
 	Let $n\geq 6$ but not $7$, then  the statement holds by \cite[Theorem 1]{Robert-Gorenstein}.
 	Let $n=7$ then by \cite[Corollary~3.8]{Bertone-smoohtable} the  $(1,7,7,1)$ locus is inside the smoothable locus, which is generically smooth, hence the statement holds.
 	\end{proof}
 
 \begin{lem}[{\cite[Proposition 3.4]{szachniewicz-nonreduce}}]\label{cubic-szach-positive}
 		Let $F$ be a very general cubic. Let $Y=\Spec(\Apolar(F))$. Then the positive tangent space $(T^1_{Y\#\emptyset})_{>0}$ vanishes.
 \end{lem}
 
\begin{stwr}[{\cite[Theorem 1.5]{szachniewicz-nonreduce}}]\label{szach-theorem}
	Let us take $B=\Apolar(F)$ such that $F$ is a very general cubic, such that $Y=\Spec(B)$. Then we get that the negative spike of $Y\#\emptyset$, $\Hilb^{<0}_{[Y\#\emptyset]}$ has the smooth type of $Y$. 
	Moreover, the Hilbert scheme is non-reduced at such point.
\end{stwr}
\begin{proof}
		Since $(T^1_Y)_{<0}=0$, by assumption of TNT, the map
		$${Y \cong \Def(*/Y) \to \Def^{<0}(* \to Y)}$$ is smooth by Corollary \ref{step1}.
		By Corollary  \ref{step1}, the map $\Def^{<0}(Y\#\emptyset /Y) \to \Def^{<0}(Y\#\emptyset \to Y)$ is smooth, again by vanishing of  $(T^1_Y)_{<0}$. \\
		$ \Def^{<0}(Y\#\emptyset \to Y) \to \Def^{<0}(Y\#\emptyset)$ is smooth by Lemma \ref{lema-gor1}, as the relative obstruction $T^2(Y,\kk)_{-d}$ vanishes by the second bullet of Definition \ref{cubic-very}. \vskip 0.2 cm
		So far, we have concluded that there is a smooth map
		 	$Y\to \Def^{<0}(Y\#\emptyset).$
		 Thereupon, we get the following Cartesian diagram up to a smooth equivalence with a section.
		 \begin{equation}\label{fractal-like-diagram}
		 	\begin{tikzcd}
		 		Y \arrow[d] \arrow[r] & \Def^{\leq 0}(Y\#\emptyset) \arrow[d]            \\
		 		* \arrow[r]           & \Def^{0}(Y\#\emptyset). \arrow[u, bend right]
		 	\end{tikzcd}
		 \end{equation}
		This implies the first part of the statement.\vskip 0.2 cm
		 For the last part, by Lemma \ref{lemma-section2}, we get that $\Def^{\leq 0}(Y\#\emptyset)$ is non-reduced.
		 The positive tangent space	$(T^1_{Y\#\emptyset})_{>0}$ vanishes since $Y$ is apolar to a very general cubic, see Lemma \ref{cubic-szach-positive}, hence by Lemma \ref{abb-smooth}, we get that $\Def(Y\#\emptyset)$ is also non-reduced.
		 As a consequence, the Hilbert scheme is also non-reduced at the point $[Y]$.
\end{proof}

\begin{remark} 
	If we assume that $\Def^0(Y\#\emptyset)$ is reduced, then we get flatness of ${\Def(Y\#\emptyset)\to \Def^0(Y\#\emptyset)}$. Then Diagram~\eqref{fractal-like-diagram} gives us \textit{ a fractal-like structure}, but it does not induce a fractal family according to Definition~\ref{fractal}, since in Diagram~\eqref{fractal-like-diagram}, the fibre is $Y$ not $Y\#\emptyset$.
	The \textit{negative spike} there naturally contains the deformed scheme $Y\#\emptyset$, but it is clearly \textbf{not} smoothly equivalent to it.
		The phenomenon showcased above has more to do with the connected sum of Gorenstein algebras (modulo the fact that it is not Gorenstein) than with this pointed case, for comparison, see Section \ref{sec-sum}.
\end{remark}

\section{Union along the point}\label{sec-glued}
In this section, we consider deformations of the following pushout scheme, called \emph{the union along the point}
$$\begin{tikzcd}
	* \arrow[d] \arrow[r] & X \arrow[d] \\
	Y \arrow[r]           & X\cup_{*} Y,
\end{tikzcd}$$
where $X$ and $Y$ are $\G_m$-pointed schemes. This is an intermediate step from deformations of pointed schemes to deformations of the connected sum.

 We assume that $\O_X$ and $\O_Y$ are generated in degree $1$. Note that from the definition ${\O_{X\cup_* Y} = \O_X \times_{\kk} \O_Y}$.
 \vskip 0.2 cm
 The pointed structures induce a closed embedding of $X\cup_* Y $ into the product $X\times Y$, making the sum a $\G_m^2$-fixed closed subscheme, in particular, we get an induced $\G_m^2$-action on the union. 
 \begin{defi}
 	When considering $\Z^2$-graded sheaf $\F$ on $X\cup_*Y$, the elements of degree $(0,k)$, for any positive integer $k\in \Z$, we call the elements of \emph{pure $y$-degree}, and analogously with $x$. Any bi-homogeneous element not of pure degree is \emph{of mixed degree}.
 \end{defi} 
 \begin{defi}
 	The $\kk$-vector space spanned by the pure $y$-degree elements is called \emph{the pure $y$-degree component}, and analogously with $x$.  The $y$-degree component admits a module structure over $\O_X$, but not over $\O_{X\cup_* Y}$.
 	The $\kk$-vector space spanned by the mixed degree elements is called \emph{the mixed degree component}.
 \end{defi}
   We want to stress that according to this notation, the elements of bi-degree $(0,0)$ are \textbf{not} of pure degree.
 The negative degree is a negative in terms of the diagonal subtorus, not to be mistaken with both weights being negative.
\vskip 0.2 cm
The purpose of this section is to prove that the map described below is smooth.  In various parts we require that neither $X$ nor $Y$ is a reduced point, see Remark~\ref{single-rmk}, Remark~\ref{single-rmk1} and Remark~\ref{single-rmk2}.
For a given Artinian local $\kk$-algebra $A$ we define a map:
\begin{align}
	\Def^{<0}(*\to X)(A) \times \Def^{<0}(*\to Y)(A) &\to \Def^{<0}(X\cup_* Y)(A), \\
	(\mathcal X,\mathcal Y) &\mapsto \mathcal X \cup_{\Spec(A)} \mathcal Y,
\end{align}
$$\begin{tikzcd}
	X \arrow[d] \arrow[r] & \mathcal X \arrow[d]           \\
	* \arrow[r]           & \Spec(A) \arrow[u, bend right]
\end{tikzcd}
\begin{tikzcd}
	Y \arrow[d] \arrow[r] & \mathcal Y \arrow[d]           \\
	* \arrow[r]           & \Spec(A) \arrow[u, bend right]
\end{tikzcd} \mapsto  \begin{tikzcd}
X\cup_* Y \arrow[d] \arrow[r] & \mathcal X\cup_{\Spec(A)} \mathcal Y \arrow[d] \\
* \arrow[r]                   & \Spec(A).                                      
\end{tikzcd}$$

 \begin{lem}
 	The map above is well-defined.
 \end{lem}
 \begin{proof}
 	The only thing to check is that $\mathcal X \cup_{\Spec(A)} \mathcal Y$ is flat over $A$.
 	Let us consider the exact sequence of $A$-modules given by the pointed structure of $\X$
 	$$0 \to \m_\X \to \O_\X \to A \to 0.$$
 	The structural morphism of $\X$ gives us a section $A\to \O_\X$, which implies:
 	$$ A\oplus \m_\X \cong \O_\X$$
 	as $A$-modules. Since $\O_\X$ is flat over $A$, the direct summand $\m_\X$ is also flat. 
 	Since ${\O_{\X\cup \Y}\cong\O_\X \times_{A}\O_\Y}$, the kernel of $\O_{\X\cup \Y}\to \O_\Y$ is isomorphic to $\m_\X$.
 	We see that $\O_{X\cup Y}$ is an extension of flat modules, therefore it is flat.
  \end{proof}
Again, the key point is Braid \eqref{braid0}, here applied for the natural maps $X\to X\cup_* Y$ and for $Y\to X\cup_* Y$.
\begin{equation}\label{braid3} \tag{$\spadesuit_5$}	
	\begin{tikzcd}[sep=small]
		T^0_{X\cup_* Y} \arrow[dd, bend right=49,dotted] \arrow[rd, dashed]                       &                                                           & T^0_X \arrow[dd,Rightarrow, bend left=49] \arrow[ld]         \\
		& T^1_{X\backslash X \to X\cup_* Y/X\cup_* Y} \arrow[rd, dashed] \arrow[ld] &                                                                   \\
		T^1_{X/ X\cup_* Y} \arrow[rd, dotted] \arrow[dd, bend right=49]             &                                                           & T^1_{X \backslash X\cup_* Y} \arrow[ld, Rightarrow] \arrow[dd, dashed, bend left=49] \\
		& T^1_{X\to X\cup_* Y} \arrow[rd, dotted] \arrow[ld, Rightarrow]    &                                                                   \\
		T^1_X \arrow[dd, Rightarrow, bend right=49] \arrow[rd]                       &                                                           & T^1_{X\cup_* Y} \arrow[dd, dotted, bend left=49] \arrow[ld, dashed]         \\
		& T^2_{X\backslash X \to X\cup_* Y/X\cup_* Y} \arrow[ld, dashed] \arrow[rd] &                                                                   \\
		T^2_{X\backslash X\cup_* Y} \arrow[rd, Rightarrow] \arrow[dd, dashed, bend right=49] &                                                           & T^2_{X/X\cup_* Y} \arrow[ld, dotted] \arrow[dd, bend left=49]             \\
		& T^2_{X\to X\cup_* Y} \arrow[ld, dotted] \arrow[rd, Rightarrow]    &                                                                   \\
		{T^2_{X\cup_* Y}}                                                                           &                                                         & {T^2_X.}                                                               
	\end{tikzcd}
\end{equation}
The line of the arguments in this section is the following. \vskip 0.2 cm
First, we will define a map ${\Def(*\to Y) \to \Def(X\backslash X\cup_* Y)}$ and show that the former is smoothly equivalent to pure $y$-degree elements of the latter. This is the content of {Lemma~\ref{step2y}}.
 Since the tangent-obstruction theory of the map $\Def(X\backslash X\cup_* Y) \to \Def(X\to X\cup_* Y)$ is given by $T^1_X,T^2_X$, see Braid \eqref{braid3}, we immediately get a similar result for the composition $${\Def(*\to Y) \to \Def(X\backslash X\cup_* Y)} \to \Def(X\to X\cup_* Y)$$
Now, for the final composition $\Def(*\to Y)\to \Def(X\cup_* Y)$, the relevant tangent-obstruction theory is $T^i_{X/X\cup_*Y}$, vanishing of pure $y$-degrees of those functors is the content of Lemma \ref{step2x}. This way, we get that $\Def(*\to Y)$ is smoothly equivalent to the pure $y$-degree part of the deformations of the union along the distinguished point, Lemma \ref{step2}. We get a symmetrical result for $\Def(* \to X)$. Now, to get the desired description of all the negative deformations this way, we need to make sure that there are no negative tangents of mixed degree; this is the content of Lemma \ref{mixed}. 
Theorem~\ref{step2-main} gives us the result announced at the beginning of the section.

\begin{lem}\label{step2y}
	There is a natural morphism 
	\begin{equation*}
		\Def(*\to Y) \to \Def(X\backslash X\cup_* Y)
		\end{equation*}
	The morphism induces an injection on tangent spaces and obstruction spaces. The image of the tangent space is the whole pure $y$-degree component.
\end{lem}
\begin{proof}
	The map
	\begin{equation}\label{step2y-eq}
		\Def(*\to Y) \to \Def(X\backslash X\cup_* Y)
		\end{equation}
	is given by taking a local Artinian algebra $A$ and a deformation $\mathcal Y$ of $Y$ over $A$, then $$(X\times \Spec A) \cup_{\Spec A} \mathcal Y$$ 
	is a nontrivial deformation over $A$ fixing the scheme $X$.
	\vskip 0.2 cm
	We ask about properties of the maps of the tangent and obstruction spaces induced by \eqref{step2y-eq}
	\begin{equation}\label{eqptY} T^i_\kk(\O_Y,\m_Y) \to T^i_{\kk}(\O_{X\cup_* Y},\m_Y),	
	\end{equation}
	for $i=1,2$.
	First, let us take
	$$ 0 \to \m_Y \to \O_Y \to \kk \to 0,$$
	the sequence coming from the distinguished point on $Y$.
	Then in the following sequence
	$$ 0 \to J \to \O_{X\cup_* Y} \to \O_X \to 0 ,$$
	the kernel is actually an $\O_Y$-module (and $\O_{X\cup_* Y}$ by the forgetful functor), it is generated in the pure $y$-degree, and thus $J$ is annihilated by the maximal ideal $\m_X\subset \O_X$.
	This way, we get
	$$\begin{tikzcd}
		0 \arrow[r] \arrow[d] & J \arrow[r] \arrow[d,equal, "Id"] & \O_{X\cup_* Y} \arrow[d, "\otimes \O_X/\m_X"] \arrow[r] & \O_X \arrow[d, "\otimes \O_X/\m_X"] \arrow[r] & 0 \arrow[d] \\
		0 \arrow[r]           & m_Y \arrow[r]                                 & \O_Y \arrow[r]                                  & \kk \arrow[r]                         & 0.          
	\end{tikzcd}$$
	The map \eqref{eqptY} is a part of a long exact sequence of $T^i$-functors, \cite[Theorem 3.5]{hardef}
	\begin{equation}\label{eqptY2} T^i_{X \cup_* Y}(\O_Y,\m_Y) \to T^i_\kk(\O_Y,\m_Y) \to T^i_{\kk}(\O_{X\cup_* Y},\m_Y).  	\end{equation}
	We prove that the pure $y$-degree part of the group $T^i_{X \cup_* Y}(\O_Y,\m_Y)$ vanishes (for $i=1,2$). 
	Consider the following exact sequence:
	$$0 \to \m_X \to \O_{X \cup_* Y} \to \O_Y\to 0. $$
	By \cite[Proposition 3.8]{hardef} $$T^1_{X \cup_* Y}(\O_Y,\m_Y) = \Hom_{\O_Y}(\m_X/\m_X^2,\m_Y),$$ this satisfies the conditions, as the group is concentrated in degree of $(-1,l)$, for $l\geq 0$.
	\vskip 0.2 cm
	We also prove that pure $y$-degree component of $T^2_{X \cup_* Y}(\O_Y,\m_Y)$ vanishes.
	The module $T^2_{X \cup_* Y}(\O_Y,\m_Y)$ is a sub-quotient of $\Hom_{X \cup_* Y}(F_2,\m_Y)$, where $F_2$ are the second minimal $\O_{X \cup_* Y}$-syzygies of $\O_Y$.
	$$ F_2 \to F_1 \to \m_X \to 0.$$
	 Each minimal syzygy of $\m_X$ as $\O_{X \cup_* Y}$-module is either of the form
	 $$ y\cdot x =0,$$
	 where $y$ is a generator of $\m_Y$, and $x$ is a generator of $\m_X$, such a syzygy is of degree $(1,1)$.
	 Or it is coming from relations of $\m_X$ in $\O_X$, then it is of pure $x$-degree, $(l,0)$, $l>0$.
	A homomorphism $\Hom_{X \cup_* Y}(F_2,\m_Y)$ of degree $(0,t)$ would send the first kind of generators to the degree $(1,1+t)$, and the second to the degree $(l,t)$. But all elements of $\m_Y$ have only pure $y$-degrees, which leads to a contradiction.
\end{proof}
\begin{remark}\label{single-rmk}
	If in Lemma \ref{step2y} we assume that $X=\{x\}$, then we get that $$\Def(*\to X) \to \Def(*\backslash X)$$ is smooth and isomorphic on the tangent spaces. It provides us with a section of the natural map $$\Def(*\backslash X)\to \Def(*\to X),$$
	proving an isomorphism of the functors.
\end{remark}
In the next lemma, we need to single out the reduced point case, see Remark \ref{single-rmk1}.
\begin{lem}\label{step2x}
	Let	$X \neq \{x\}$.
	$\Def(X/ X\cup_* Y)$ has no tangents nor obstructions lying in the pure $y$-degree.
\end{lem}
\begin{proof}
	The tangent-obstruction theory of $\Def(X/ X\cup_* Y)$ is $T^{i}_{X\cup_* Y}(\O_X,\O_X)$, for $i=1,2$, see Definition~\ref{defi-XY}. Let us recall that we have the following exact sequence of $\O_{X\cup_* Y}$-modules
	$$ 0 \to \m_ Y \to \O_{X\cup_* Y} \to \O_{X} \to 0,$$
	$$ 0 \to \m_ X \to \O_{X\cup_* Y} \to \O_{Y} \to 0.$$
	By the usual argument on the Hilbert schemes, it is enough to show the vanishing of pure $y$-degree part for:
	$$\Hom_{\O_X \times_\kk \O_Y }(\m_Y,\O_X).$$
	This is an $\O_Y$-module, since it is annihilated by $\m_X$.
	The module $\m_Y$ is generated in degree $(0,1)$. Let us assume that there is a homomorphism $\phi \in \Hom_{\O_X \times_\kk \O_Y }(\m_Y,\O_X)$ of pure $y$-degree. The generators should be mapped to degree zero in $\O_X$, which is $\kk\subset \O_X$. If $\phi$ is non-zero, then there are invertible elements in the image, but this contradicts the $\O_Y$-structure of $\m_Y$. \vskip 0.2 cm
	We also need to consider the obstruction module
	$$T^{2}_{X\cup_* Y}(\O_X,\O_X) \subset \Ext^1_{\O_{X\cup_* Y} }(\m_Y,\O_X).$$
	We consider the resolution of $\m_Y$ as $\O_{X\cup_* Y}$-module
	$$F_3 \xrightarrow{d_2} F_2 \xrightarrow{d_1} F_1 \to \m_Y \to 0.$$
	As in Lemma \ref{step2y} (but for $\m_Y$ instead of $\m_X$), the minimal second syzygies are in degrees $(1,1)$ or $(0,l)$, for $k,l>0$. \vskip 0.2 cm
	 Assume that we have a homomorphism $\psi \in \Hom(F_2,\O_X)$ of pure $y$-degree that provides us a class in $\Ext^1_{\O_X \times_\kk \O_Y }(\m_Y,\O_X)$.
	 \vskip 0.2 cm
	 The second kind of syzygies goes with $\psi$ to the degree $(0,0)$, since $\O_X$ is generated as an algebra by pure $x$-degrees. If the image of those syzygies is non-zero then $\psi$ sends one of those syzygies to the unit.
	 If $\psi \colon F_2\to \O_X$ lifts to an element in $\Ext^1_{\O_X \times_\kk \O_Y }(\m_Y,\O_X)$ group, then the following diagram commutes \begin{equation}\label{3syz-eq}
	 \begin{tikzcd}
	 	F_3 \arrow[d, "d_2"] \arrow[rd, "0"] &      \\
	 	F_2 \arrow[r, "\psi"]                & \O_X.
	 \end{tikzcd}
	 \end{equation}
	 Let us say that a syzygy $s$ goes to $1$ with $\psi$. For $x \in \O_X \subset  \O_{X\cup_* Y}$, any non-zero element of pure $x$-degree, $xs$ is mapped by $\psi$ to $x\in \O_X$, which is non-zero. \\
	 Since $d_1(s)\in \m_Y F_1$, consequently $xd_1(s)=0$ (note that we are in category of $\O_{X\cup_* Y}$-modules), and there is a $3$-syzygy with the image $xs$ . Such a $3$-syzygy contradicts the commutativity of Diagram \eqref{3syz-eq}. Consequently, the second kind of syzygies are annihilated by $\psi$.
	 \vskip 0.2 cm
	 Let us consider the submodule of $F_2$ generated by the syzygies of the first kind, those of degree $(1,1)$.
	 The quotient of this submodule by (the restriction of) the image of $d_2$ is isomorphic to $\bigoplus \m_X[0,-1]$, where the sum is indexed by the homogeneous generators of $\m_Y$.
	The set of elements in $$\Hom_{\O_X \times_\kk \O_Y }(\m_X[0,-1],\O_X)$$ of pure $y$-degree is the set of degree zero elements of 
	$$\Hom_{\O_X \times_\kk \O_Y }(\m_X,\O_X).$$
	There is only one such homogeneous element, coming from the embedding. But such a homomorphism factors through the map $d_1$. Thereupon, giving the trivial class in the extensions group.
\end{proof}

\begin{remark} \label{single-rmk1}
We comment on the necessity of the assumption that $X\neq \{x\}$. Assume that $X$ is a reduced point, then we get a functor $\Def(*/Y)$, this is the Hilbert scheme of one point on $Y$, then the tangent space lies in the pure $y$-degree, see Proposition \ref{pt-in-Y}. This contradicts the statement as long as $Y$ is not a reduced point as well.
\end{remark}

\begin{lem} \label{step2}
		Assume that $X\neq \{x\}$.
		Then the map $\Def^{<0}(* \to Y)  \to \Def^{<0}(X\cup_* Y)$ is injective on tangents and obstructions, and surjects onto $y$-pure degree tangents.
\end{lem}
\begin{proof}
	Lemma \ref{step2y} implies the statement for the map $$\Def^{<0}(*\to Y) \to \Def^{<0}(X\backslash X\cup_* Y).$$ Since a tangent-obstruction theory of the map
	$$ \Def(X\backslash X\cup_* Y) \to \Def(X\to X\cup_* Y)$$ is given by $T^0_X,T^1_X$, see Braid~\eqref{braid3}. Therefore, the tangent-obstruction theory is concentrated in pure $x$-degrees, we get the statement for the composition
	$$\Def^{<0}(*\to Y)  \to \Def^{<0}(X\to X\cup_* Y).$$
	By Braid~\eqref{braid3} the tangent-obstruction theory of the map $$\Def(X\to X\cup_* Y) \to \Def(X\cup_* Y),$$ 
	is given by the tangent-obstruction theory of $\Def(X/ X\cup_* Y)$, by Lemma \ref{step2x} it is concentrated in pure $x$-degrees.
	 Thus the composition
	 $$\Def^{<0}(*\to Y) \to \Def^{<0}(X\to X\cup_* Y) \to \Def(X\cup_* Y),$$
	 again satisfies the conditions, the fact that this composition is the map from the statement follows from the commutativity of Braid~\eqref{braid3}.
\end{proof}

\begin{remark}\label{single-rmk2}
	We want to stress the necessity of $X$ not being a reduced point. 
	If in Lemma~\ref{step2} we assume that $X=\{x\}$, then the statement says that
	$\Def^{<0}(* \to Y)  \to \Def^{<0}(Y)$
	is smooth. In general, this is not true. Corollary \ref{iso-pointed} states that under certain open conditions ${\Def(*/Y)\to \Def^{<0}(*\to Y)}$ is smooth.
	Since the composition of those maps factors through a point, as presented in the diagram below
	$$	\begin{tikzcd}
		\Def(*/Y) \arrow[r] \arrow[d] & \Def^{<0}(Y) \arrow[d] \\
		* \arrow[r]                   & \Def^{<0}(Y).          
	\end{tikzcd}$$
	We get that $\Def(*/Y)$ is smooth. But, by Proposition \ref{pt-in-Y} this functor is prorepresented by the distinguished point of $Y$, that implies that $Y$ is smooth at the distinguished point.
\end{remark}
\begin{defi}\label{defi-level}
	An Artinian graded local algebra $(B,\m)$ is a \emph{$d$-level algebra} if the socle $B$ is equal to the $d$-th graded part $B_d$. A finite local $\G_m$-scheme $Y$ is \emph{$d$-level} if $\O_Y$ is a $d$-level algebra.
\end{defi}

\begin{lem}\label{mixed}
	Let us assume that both $X$ and $Y$ are spectra of $d$-level algebras.
	There are no negative mixed tangents in $T^1_{X\cup_* Y}$.
\end{lem}
\begin{proof}
	Let us choose a resolution $$0\to J \to R \to \O_{X\cup_{*} Y} \to 0,$$
	where $R=\kk[x_1,\ldots, x_{n_1},y_1,\ldots ,y_{n_2}]$ is a polynomial ring formally generated by minimal homogeneous generators of $\O_X$ and $\O_Y$.
	By \cite[Proposition 3.10]{hardef}, there is the following surjection:
	$$ \Hom(J/J^2,\O_{X \cup_* Y}) \to T^1_{X\cup_* Y}\to 0.$$
	We compute the Hilbert function of the former module.
	Let us take a bi-homogeneous element $\phi \in \Hom(J/J^2,\O_{X \cup_* Y}) $.
	The ideal $J$ is generated by the forms $x_iy_j$ and by some pure $x$-degree and some pure $y$-degree polynomials.
	Since there is no elements of mixed degree in the ring $\O_{X\cup_* Y}$, the degree of $\phi$ is either $(-1,-1)$ or one of the weights is positive.
	First, assume that the degree of $\phi$ is $(-1,-1)$. In such a case, only generators of the form $x_iy_j$ can survive. But then $\phi(x_iy_j^N) = 0 $ for a minimal positive integer $N$ such that $y_j^N \in J$, hence $\phi(x_iy_j)y_j^{N-1}=0$, but since $\phi(x_iy_j)$ is of degree $(0,0)$ it either vanishes or it is invertible, since $N$ was minimal, we get the vanishing.
	\vskip 0.2 cm
	Now let us assume that $\phi$ is of degree $(k,-l)$ for integers satisfying $k,l>1$ and $k-l<0$.
	Let us take a minimal generator $g \in J$ of pure $y$-degree, if the $y$-degree of $g$ is smaller than $l$ then $\phi(g)$ vanishes; thus its $y$-degree is greater or equal to $l$. Since $x_ig\in (x_iy_j)_{1\leq i\leq n_1,1\leq j\leq n_2}$, and $\phi(x_iy_j)=0$, the form $x_i\phi(g)=\phi(x_ig)$ vanishes, hence $\phi(g)$ is of degree $(k,0)$ and is an element of socle of $\O_X$, since $\O_X$ is a $d$-level algebra we get $k=d$. Since $\O_Y$ is a $d$-level algebra, the $y$-degree of $g$ is at most $d$, in particular $l<d$ and as a consequence, $k-l$ cannot be negative.
\end{proof}
Having all necessary prerequisites, we prove the result on the negative deformations of the union along the point, which is announced at the beginning of the section. 
 We also show that the union along point inherits the vanishing of tangent space in degrees $\leq -1$, which is useful in the next section, see Theorem~\ref{thm-negative}.
\begin{twr} \label{step2-main} \label{pointed-=union-}
	Assume that $X,Y$ are pointed finite $\G_m$-schemes which are spectra of $d$-level algebras, neither is a reduced point.
	Then the morphism:
	$$ \Def^{<0}(* \to X) \times \Def^{<0}(* \to Y) \to \Def^{<0}(X\cup_* Y)$$
	is smooth and induces isomorphism on the tangent spaces.	
	Moreover, if the groups $(T^1_X)_{<-1}$, $(T^1_Y)_{<-1}$ vanish, then
	$$ (T^1_{X\cup_*Y})_{<-1}=0.$$
\end{twr}
\begin{proof}
	By Lemma \ref{step2} we get that the map is injective on tangents and obstructions. Moreover, it surjects onto pure $x$-degree and pure $y$-degree tangents.
	The only thing left is to show that there are no mixed negative tangents. This is exactly Lemma \ref{mixed}.
	For the second part of the statement,
	the negative tangents are either coming from $T^1_{*\to Y}$ or $T^1_{*\to X}$, but Corollary~\ref{pre-fractal}, they both vanish in degrees $< -1$.
\end{proof}


\begin{stwr} \label{fractal-union}
	Let $X,Y$ as in the assumptions of Theorem \ref{step2-main} with the assumption that the groups $(T^1_X)_{<-1}$, $(T^1_Y)_{<-1}$ vanish, then we get the following diagram \begin{equation} \label{diagram511}
\begin{tikzcd}
	X\times Y \arrow[d] \arrow[r] & \Def^{<0}(X\cup_* Y) \arrow[d] \\
	* \arrow[r]                   & \Def^{<0}(X) \times \Def^{<0}(Y)
\end{tikzcd} 	
\end{equation}
to be Cartesian up to smooth equivalence.
\end{stwr}
\begin{proof}
	By Corollary \ref{pre-fractal} we get the following Cartesian diagram up to smooth equivalence
	$$\begin{tikzcd}
		\Def^{<0}(*/X) \times \Def^{<0}(*/Y)  \arrow[d] \arrow[r] & \Def^{<0}(*\to X) \times \Def^{<0}(*\to Y) \arrow[d] \\
		* \arrow[r]                   & \Def^{<0}(X)\times \Def^{<0}(Y).               
	\end{tikzcd}$$
	But by Theorem \ref{step2-main} we get that $\Def^{<0}(*\to X) \times \Def^{<0}(*\to Y) $ and $\Def^{<0}(X\cup_* Y) $ are smoothly equivalent.
\end{proof}

\begin{remark}
	One can consider a fractal-like statement for the Hilbert scheme at the point associated to the union of schemes along a point. A fractal statement does not hold, since the negative fibre (see Diagram \eqref{diagram511}) will be the product $X\times Y$, not the union $X\cup_* Y$ taken at the beginning. A fractal-like statement also does not hold. In the last part of the proof of fractalness, we use Lemma \ref{abb-smooth}. In the case of the union along a point, we have no guarantee that the positive tangents will abstractly vanish, even if both $\Def^{> 0}(X)$ and $\Def^{> 0}(Y)$ do vanish.
\end{remark}

Let us state another interesting corollary on non-reducedness. The point is that smoothness from Theorem \ref{step2-main} implies that the functor of deformations of the union drags behind the fractal structure of $\Def(*\to X)$ and $\Def(*\to Y)$.
\begin{stwr}\label{non-reduced-union}
	Let $X,Y$ be a pair of uncleavable scheme satisfying the assumptions of Theorem~\ref{step2-main} then ${\Def^{<0}(X\cup_* Y)}$ is non-reduced.
\end{stwr}
\begin{proof}
	By Lemma \ref{pointed-nonreduced} the functors $\Def^{<0}(* \to X) $ and $\Def^{<0}(* \to Y)$ are non-reduced, hence the statement holds.
\end{proof}

\section{Connected sum of Gorenstein algebras} \label{sec-sum}
In this section, at last, we focus on the Gorenstein case. We move from the union along the point $X\cup_* Y$ to the connected sum $X\#Y$, first introduced in \cite[Section 2. Connected sums]{MR2929675}. We use its description in terms of apolar algebras developed in \cite{IARROBINO-connected}, see Definition~\ref{conn-defi}.
Our main goal is to transform the fact (Theorem \ref{step2-main}) that the morphism
$$ \Def^{<0}(* \to X) \times \Def^{<0}(* \to Y) \to \Def^{<0}(X\cup_* Y)$$
is smooth and isomorphic on the tangent spaces, to a description of negative deformations of the connected sum
$$\Def^{<0}(X\#Y).$$
This functor does not admit a smooth map from the functor $\Def^{<0}(X\cup_* Y)$, but it is equivalent to a hypersurface cut out by a one equation coming from the socle; this we prove in Theorem~\ref{thm-negative}.
\begin{set}\label{set-gor}
	Let
	$$X=\Spec \Apolar(F),$$
	$$Y=\Spec\Apolar(G),$$
	$F\in P_x$ and $G\in P_y$, $F,G$ homogeneous polynomials of degree $d\geq 3$.
	such that
	$$ (T^1_X)_{<-1}=(T^1_Y)_{<-1}=0$$
	 is satisfied.
	We add two conditions
	\begin{itemize}
		\item 	The following second cotangent modules vanish	$$T^2(\O_X,\kk)_{<-d}=T^2(\O_Y,\kk)_{<-d}=0.$$
		\item The ideals $I_X=\Ann(F)\subset S_x$ and $I_Y=\Ann(G)\subset S_y$ are generated in {the~degrees~${\leq d-1}$}.
	\end{itemize}

\end{set}

\begin{remark}
		The conditions of being $d$-level from Theorem \ref{step2-main} is satisfied by the condition of being an apolar algebra of a polynomial of degree $d$.
\end{remark}

Since $\O_{X\cup_* Y}$ admits a socle of length $2$, spanned by the socles of $\O_X$ and $\O_Y$, the union along the point is not Gorenstein. To get a Gorenstein algebra, we identify the socles.
\begin{defi}[{\cite[Theorem 4.6]{IARROBINO-connected}}] \label{conn-defi}. The connected sum of Gorenstein finite schemes ${X=\Spec\Apolar(F)}$ and $Y=\Spec\Apolar(G)$ is given by the formula
	$$X\#Y = \Spec \Apolar(F+G)=\Spec \frac{\O_{X\cup_{*}Y}}{f-g}$$
	where $f\circ F=1$ and $g\circ G=1$, are homogeneous elements generating socles.

\end{defi}
\begin{notation}
	When $X$ and $Y$ are treated as bases of deformations, the parameters of $\O_X$ and $\O_Y$ are denoted by $t_x$ and $t_y$, respectively. We do that to distinguish them from generators of the special fibre algebra.
\end{notation}
Let us take a closed subscheme $Z\subset X\cup_* Y$ given by a homogeneous (but not necessarily bi-homogeneous) equation $h$ in the socle.
\begin{equation} \label{eq-Z}
	Z:=\Spec(\O_{X\cup_* Y}/(h)) \subset X\cup_* Y.
\end{equation}
The main example of $Z$ is $X\#Y$, but since the connected sum does not admit $\G_m^2$-action, we need more flexibility and we also consider $X\#\emptyset \cup_* Y$, which in turn has a $\G_m^2$-action.
Once again, our main technical tool will be Braid~\eqref{braid0}, in this section, we apply it for the embedding $Z \to X\cup_* Y$.

\begin{equation}\label{braid4}	\tag{$\spadesuit_6$}
	\begin{tikzcd}
		T^0_{X\cup_* Y} \arrow[dd, bend right=49,dotted] \arrow[rd, dashed]                       &                                                           & T^0_{Z} \arrow[dd,Rightarrow, bend left=49] \arrow[ld]         \\
		& T^1_{{Z}\backslash {Z} \to X\cup_* Y/X\cup_* Y} \arrow[rd, dashed] \arrow[ld] &                                                                   \\
		T^1_{{Z}/ X\cup_* Y} \arrow[rd, dotted] \arrow[dd, bend right=49]             &                                                           & T^1_{{Z} \backslash {X}\cup_* Y} \arrow[ld, Rightarrow] \arrow[dd, dashed, bend left=49] \\
		& \mathbf{T^1_{{Z}\to X\cup_* Y}} \arrow[rd, dotted] \arrow[ld, Rightarrow]    &                                                                   \\
		\mathbf{T^1_{Z}} \arrow[dd, Rightarrow, bend right=49] \arrow[rd]                       &                                                           & \mathbf{T^1_{X\cup_* Y}} \arrow[dd, dotted, bend left=49] \arrow[ld, dashed]         \\
		& T^2_{{Z}\backslash {Z} \to X\cup_* Y/X\cup_* Y} \arrow[ld, dashed] \arrow[rd] &                                                                   \\
		T^2_{{Z}\backslash X\cup_* Y} \arrow[rd, Rightarrow] \arrow[dd, dashed, bend right=49] &                                                           & T^2_{{Z}/X\cup_* Y} \arrow[ld, dotted] \arrow[dd, bend left=49]             \\
		& \mathbf{T^2_{{Z}\to X\cup_* Y}} \arrow[ld, dotted] \arrow[rd, Rightarrow]    &                                                                   \\
		\mathbf{T^2_{X\cup_* Y}}                                                                           &                                                           & \mathbf{T^2_{{Z}}.}                                                               
	\end{tikzcd}
\end{equation}
	We proceed as follows. We compare $\Def^{<0}(Z)$ and $\Def^{<0}(X\cup_* Y)$, we do so through the functor of deformation of the given embedding.
	\begin{equation}\label{dach}
\begin{tikzcd}
		& \Def^{<0}(Z\to X\cup_* Y) \arrow[ld, dotted] \arrow[rd, Rightarrow] &                      \\
		\Def^{<0}(Z) &                                                                     & \Def^{<0}(X\cup_* Y).
	\end{tikzcd}	\end{equation}
	Thereupon, we focus on the maps between those $T^i$ functors that are written in Braid~\eqref{braid4} with the bold font.
	The left arrow of Diagram~\eqref{dach} is smooth; this is the content of Lemma \ref{step3l}.
	The right arrow is obstructed and its obstructness is described in Lemma \ref{step3}.\\
	Theorem \ref{pre-thm} and Theorem \ref{thm-negative} conclude the argument by showing the result announced at the beginning of the section. 
\begin{lem} \label{step3l}
	The map $\Def^{<0}(Z\to X\cup_* Y) \to \Def^{<0}(Z)$
	is smooth and isomorphic on the tangent spaces.
\end{lem}
\begin{proof}
	Corollary \ref{lem-tech-braid} states that it is enough to show that a relative tangent-obstruction theory vanishes, as we see in Braid \eqref{braid4}.
	That is 
		$$(T^i_{Z \backslash X \cup_* Y})_{<0}=0, $$
	for $i=1,2$. Recall that the definition of $Z$, see Formula \eqref{eq-Z}, by Definition \ref{defi-XY} we get the following identities
	$$T^i_{Z \backslash X \cup_* Y}= T^i_\kk(\O_{X\cup_*Y},(h))= T^i_\kk(\O_{X\cup_* Y},\kk)[d].$$
	Hence, the negative part is isomorphic to
	$$ T^i_\kk(\O_{X\cup_* Y},\kk)_{<-d}.$$
	Since $X$ and $Y$ are spectra of $d$-level algebras, by Theorem \ref{step2-main}, the tangent part vanishes.
	It remains to prove that 
	 $T^2_\kk(\O_{X\cup_* Y},\kk)_{<-d}$
	 vanishes.

	\vskip 0.2 cm
	
	Let us consider the minimal free bi-graded resolution of $\O_{X \cup_* Y}$ over a polynomial ring $R$ generated 
	by the minimal generators of $\O_X$ and $\O_Y$.
	$$F_3 \xrightarrow{d_2} F_2 \xrightarrow{d_1} F_1 \to R \to  \O_{X \cup_* Y} \to 0.$$
	Recall that 	$$ T^2_\kk(\O_{X\cup_* Y},\kk) \subset \Ext_R^2(\O_{X\cup_* Y},\kk).$$
	\\
	The pure degree components of $T^2_\kk(\O_{X\cup_* Y},\kk)$ are either coming from $T^2_\kk(\O_{X},\kk)$ and $T^2_\kk(\O_{Y},\kk)$, thus we need to focus on the mixed degree elements. This can be seen as follows, we lift a pure $x$-degree element to a homomorphism $\phi \in \Hom_R(F_2,\kk)$ of the same degree. Such a homomorphism annihilates all syzygies that are not pure $x$-degree. Hence, we obtain a candidate for a class in $T^2(\O_X,\kk)$. By Setting~\ref{set-gor},
	$T^2_\kk(\O_{X},\kk)_{<-d}$ vanishes, thus we can construct a homomorphism $\psi\in \Hom(F_1,\kk)$ such that precomposed with $d_1$ gives $\phi$. Therefore, the pure degree components vanish in degrees $< -d$.\vskip 0.2 cm
	Now we consider the mixed degree component of $T^2_\kk(\O_{X\cup_* Y},\kk)$. We choose a homomorphism $\phi \in \Hom_R(F_2,\kk)$ as before. Since the elements of the codomain $\kk$ are of bi-degree $(0,0)$, all the elements of pure degree are mapped to $0$ by $\phi$.
 	The minimal homogeneous mixed degree syzygies are of one of two kinds.
	The first kind consists of syzygies of the ideal $\mathbf x \mathbf y$ (the product of maximal ideals of $\O_X$ and $\O_Y$). Those syzygies are linear, hence of degree $3$, a homomorphism of degree strictly smaller than $-d\leq -3$ has to annihilate them.
	\vskip 0.2 cm
	The second kind consists of syzygies expressing the fact that a pure $x$-degree  (respectively $y$) generator times any element of pure $y$-degree (resp. $x$) is in the ideal $\mathbf x \mathbf y$. Since, by Setting \ref{set-gor}, the degrees of generators are not greater than $d$, their bi-degree can be maximally $(1,d-1)$ or $(d-1,1)$. As a consequence, a homomorphism of degree strictly lower than $-d$ has to annihilate them all.
\end{proof}

	Now we study a relative tangent-obstruction theory of map $\Def(Z\to X\cup_* Y) \to \Def(X\cup_* Y)$, see Braid \eqref{braid4}. Recall that $Z$ is defined as the vanishing locus of equation $h$, see Formula \eqref{eq-Z}.
\begin{lem}\label{step3}
		The following cotangent module vanishes
		$$(T^1_{Z / X \cup_* Y})_{<0} =0. $$
		The second cotangent module $$(T^2_{Z / X \cup_* Y})_{<0} $$ is generated in degree $-d$ by one element. If $h$ is bigraded with $(d_1,d_2)$ then the obstruction has the bi-degree $(-d_1,-d_2)$.
\end{lem}
\begin{proof}
	The tangent space is a one-dimensional $\kk$-linear space generated by an element of degree zero.
	$$ T^1_{Z / X \cup_* Y}= \Hom_{\O_{X \cup_* Y}}((h),\O_{X\#Y })=\Hom_{\O_{X \cup_* Y}}(\kk,\O_{Z })[d]=\kk[-d][d]=\kk.$$
	We compute the obstruction space
	$$T^2_{Z / X \cup_* Y}= \Ext^1_{\O_{X \cup_* Y}}((h),\O_{Z })=\Ext^1_{\O_{X \cup_* Y}}(\kk,\O_{Z})[d],$$
	this module is concentrated in degree $-d$ (if $h$ is bi-graded then $(-d_1,-d_2)$). We compute 
	$$ \Hom_{\O_X}(\Tor_1^{\O_Y}(\kk,\O_{Z}),\O_{Z}) = \Ext^1_{\O_{X\cup_*Y}}(\kk,\O_{Z}).$$
	By the symmetry of the $\Tor$ functor, we can change the order of the arguments:
	$$ \Tor_1^{\O_{X\cup_*Y}}(\kk,\O_{Z}) = \Tor_1^{\O_{X\cup_*Y}}(\O_{Z},\kk)$$
	and look at the projective resolution of $\O_{Z}$ over $\O_{X\cup_* Y}$:
	$$\begin{tikzcd}
		F_2 \arrow[r] \arrow[d]  & \O_{X\cup_*Y} \arrow[r, "h"] \arrow[d] & \O_{X\cup_*Y} \arrow[r, "1"] \arrow[d]      & \O_{Z} \arrow[r] \arrow[d, "\otimes_{\O_{X\cup_*Y}}\kk"] & 0 \\
		F_2\otimes \kk \arrow[r] & \kk \arrow[r, "0"]            & \kk \arrow[r, no head,"Id"] & \kk \arrow[r]                              & 0.
	\end{tikzcd}$$
	Hence, we conclude that $\Tor_1^{\O_{X\cup_*Y}}(\O_{Z},\kk)=\kk$. That proves this part of the argument.
\end{proof}
If we take $h$, see Formula \eqref{eq-Z}, to be of degree $(d,0)$, then $h$ arises as a lift of a generator of the socle of $\O_X$, in particular $\O_Z=\O_{X\cup_* Y}/(h)$ is isomorphic to $\O_{X\#\emptyset}\times_\kk \O_Y$, which in turn is isomorphic to the structural sheaf of $X\#\emptyset \cup_* Y$. Recall from Definition \ref{conn-defi} that $f$ is a generator of the socle of $\O_X$, any $h$ of degree $(d,0)$ is $f$ rescaled by a scalar.
\begin{twr}\label{pre-thm}
	The good deformation functor $\Def^{<0}(X\#\emptyset\cup_*Y)$ is smoothly equivalent 
	to a closed hypersurface of $\Def^{<0}(X\cup_*Y)$ cut by an equation in bi-degree $(d,0)$.
	Moreover there is a diagram
	$$ \begin{tikzcd}
		\Spec(\O_X\otimes\O_Y/(f(t_x))) \arrow[d] \arrow[r] & \Def^{<0}(X\#\emptyset \cup_* Y) \arrow[d]      \\
		* \arrow[r]                                              & \Def^{<0}(X)\times \Def^{<0}(Y),
	\end{tikzcd}$$
	 Cartesian up to smooth equivalence.
	In particular, if $X$ and $Y$ are TNT algebras, then we get  a smooth equivalence
		 $$\Spec\frac{\O_X\otimes\O_Y}{f(t_x)} \to \Def^{<0}(X\#\emptyset \cup_* Y).$$
\end{twr}
\begin{proof}
	By Corollary \ref{fractal-union}, there exists the following Cartesian square up to smooth equivalence
$$\begin{tikzcd}
	P \arrow[d] \arrow[r]         & \Def^{<0}(X\#\emptyset\cup_* Y) \arrow[d] \\
	X\times Y \arrow[r] \arrow[d] & \Def^{<0}(X\cup_* Y) \arrow[d]    \\
	* \arrow[r]                   & \Def^{<0}(X)\times\Def^{<0}(Y).    
\end{tikzcd}
$$
	In the diagram above, we abuse the notation and we write $\Def^{<0}(X\# Y)$ in place of ${\Def^{<0}(X\# Y \to X\cup_* Y)}$, as by Lemma \ref{step3l} there is a smooth map between those functors.
	The pullback map $P\to X\times Y$ is an isomorphism on the tangent spaces and has one obstruction in  degree $(d,0)$, see Lemma~\ref{step3}.
If we prove that this map is not smooth, then it is a closed embedding given by one equation of degree $(d,0)$. Consequently, it is isomorphic to the spectrum of algebra $(\O_X\otimes \O_Y)/(f)$. This is exactly the statement of the theorem.
\vskip 0.2 cm
Since, by the second part of Lemma \ref{lema-gor1}, the scheme $X\#\emptyset$ satisfies the vanishing ${(T^1_{X\#\emptyset})_{<-1}=0}$, we get the following Cartesian up to smooth equivalence diagram
$$
\begin{tikzcd}
	\Def(*/X\#\emptyset)  \arrow[d] \arrow[r] & \Def^{<0}(*\to X\# \emptyset)    \arrow[d]  \\
	* \arrow[r]                                      & \Def^{<0}(X\# \emptyset),                    &                         
\end{tikzcd}$$
since $X\#\emptyset$ and $Y$ satisfy the conditions of Lemma \ref{step2}, the map 
$$\Def^{<0}(*\to X\# \emptyset) \to \Def^{<0}(X\#\emptyset \cup_* Y )$$
is unobstructed and surjective on pure $x$-degrees, thus if we restrict to pure $x$-degrees we get the following diagram
$$\begin{tikzcd}
	P_X \arrow[d] \arrow[r] & \Def^{<0}(*\to X\#\emptyset) \arrow[d] \\
	X \arrow[r] \arrow[d]   & \Def^{<0}(* \to X) \arrow[d]   \\
	* \arrow[r]             & \Def^{<0}(X),                  
\end{tikzcd}$$
where $P_X$ is the $x$-graded part of $P$. To be precise $P_X=P^{\{0\}\times \G_m}$ is the equivariant functor with respect to the second component $\G_m$ (coming from $Y$).
The composed map ${\Def(*/X\#\emptyset)^{<0}\to \Def(X)^{<0}}$ is trivial, since by Lemma \ref{lema-gor1} the functor $\Def^{<0}(X\#\emptyset\backslash X)$ is trivial, as it's tangent and obstruction spaces vanish.
Therefore, we get a map $$X\#\emptyset=\Def(*/X\#\emptyset) \to P_X \to X.$$
If $P\to X\times Y$ is smooth and isomorphic on tangent spaces, then so is $P_X \to X$, and then the map $X\# \emptyset \to X$ is injective on tangent spaces and unobstructed, which is false.
\end{proof}
The final argument is derived from the previous one. The initial part is identical, but then we need to employ the former result.
\begin{twr} \label{thm-negative}
	The good deformation functor $\Def^{<0}(X\#Y)$ is smoothly equivalent to a closed hypersurface of $\Def^{<0}(X\cup_* Y)$ cut by an equation in degree $d$. 
	 Moreover, we get the diagram	
	$$ \begin{tikzcd}
	 	\Spec(\O_X\otimes\O_Y/(f(t_x)-g(t_y))) \arrow[d] \arrow[r] & \Def^{<0}(X\# Y) \arrow[d]      \\
	 	* \arrow[r]                                              & \Def^{<0}(X)\times \Def^{<0}(Y),
	 \end{tikzcd}$$
	 to be Cartesian up to smooth equivalence.
	 In particular, if $X$ and $Y$ are TNT, algebras then we get a smooth map
	 $$\Spec\frac {\O_X\otimes\O_Y}{f(t_x)-g(t_y)} \to \Def^{<0}(X\# Y).$$
\end{twr}
\begin{proof}
	We proceed similarly to Theorem \ref{pre-thm}.
		We take the following Cartesian square.
	$$\begin{tikzcd}
		P \arrow[d] \arrow[r]         & \Def^{<0}(X\# Y) \arrow[d] \\
		X\times Y \arrow[r] \arrow[d] & \Def^{<0}(X\cup_* Y) \arrow[d]    \\
		* \arrow[r]                   & \Def^{<0}(X)\times\Def^{<0}(Y).    
	\end{tikzcd}
	$$
	In the diagram above, we abuse the notation and we write $\Def^{<0}(X\# Y)$ in place of ${\Def^{<0}(X\# Y \to X\cup_* Y)}$, as by Lemma \ref{step3l} there is a smooth map between those functors.
	The pullback map $P\to X\times Y$ is an isomorphism on tangent spaces and has one obstruction in degree $d$.
	If we prove that this map is not smooth, then it is a closed embedding given by one non bi-homogeneous equation of degree $d$. As a consequence, it is isomorphic to the spectrum of $\O_X\otimes \O_Y/(f-g)$. That is exactly the statement of the theorem.
	\vskip 0.2 cm
	Let us assume that the pullback map $P\to X\times Y$ is smooth. Let us take the image of maximal family from $X\times Y$ in $\Def^{<0}(X\cup_*Y)$, this is a local algebra $(B,\mathfrak n)$ over $(A,\m)=(\O_X\otimes \O_Y,\m)$ which over the closed point specializes to $B/(\m B)=\O_{X\cup_*Y}$, see Theorem \ref{step2-main}.
	Since $\O_{X\#Y}$ is given as a quotient $\O_{X \cup_* Y}/(f-g)$, a lift of $B$ to $\Def^{<0}(X\#Y)$ is given by a lift $\tilde f - \tilde g$ of the equation $f-g$ (we assume that $\tilde f$ is a bigraded lift of $f$ and the same for $\tilde g$ and $g$). By flatness, we obtain the following map of exact sequences
	$$\begin{tikzcd}
		0 \arrow[r] & \Ann(\tilde f - \tilde g) \arrow[r] \arrow[d] & B \arrow[r, "\tilde f -\tilde g"] \arrow[d] & B \arrow[r] \arrow[d]    & B/(\tilde f - \tilde g) \arrow[r] \arrow[d] & 0 \\
		0 \arrow[r] & \Ann(f-g) \arrow[r]                           & \O_{X\cup_* Y} \arrow[r]                    & \O_{X\cup_* Y} \arrow[r] & \O_{X\#Y} \arrow[r]                         & 0.
	\end{tikzcd}$$
	The following identities hold $$\Ann(f-g)=\Ann(f)=\Ann(g)=\frac{\mathfrak n}{\m B},$$ since those are equations from the socle. Note that $\mathfrak n$ is generated in degree $1$, so in the bi-degrees $(1,0)$ and $(0,1)$.
	The annihilator $\Ann(\tilde f - \tilde g)$ is a graded module (not necessarily bi-graded).
	Let us take $b\in \Ann(\tilde f - \tilde g) \subset B$ it can be decomposed into bi-graded pieces, since $\deg(b)=1$ then $b=b_x+b_y$ with $\deg(b_x)=(1,0)$ and $\deg(b_y)=(0,1)$.
	$$ 0=(b_x+b_y)(\tilde f - \tilde g)= b_x\tilde f + b_y\tilde f- b_x\tilde g- b_y\tilde g,$$
	where every element in the sum has a different degree, thus $(b_x+b_y)(\tilde f)=0$, as a consequence $$\Ann(\tilde f - \tilde g) \subset \Ann(\tilde f).$$
	Now, we consider the problem of $A$-flatness of $B/\tilde f $. By the local criterion for flatness, we need to compute if $\Tor^1_A(A/\m,B/(\tilde f))$ vanishes.
		$$\begin{tikzcd}
		0 \arrow[r] & \Ann(\tilde f ) \arrow[r] \arrow[d] & B \arrow[r, "\tilde f"] \arrow[d] & B \arrow[r] \arrow[d]    & B/(\tilde f ) \arrow[r] \arrow[d] & 0 \\
		& \Ann(\tilde f )\otimes A/\m \arrow[r]                           & \O_{X\cup_* Y} \arrow[r]                    & \O_{X\cup_* Y} \arrow[r] & \O_{X\#\emptyset \cup_* Y} \arrow[r]                         & 0.
	\end{tikzcd}$$
	Since $\Ann(\tilde f - \tilde g) \subset \Ann(\tilde f)$ we get that $\Ann(\tilde f)\otimes A/\m$ contains lifts of $\mathfrak n/\m B$. That implies the desired vanishing. Consequently, we have shown that assuming smoothness of $P\to X\times Y$ implies that $$\Def^{<0}(X\#\emptyset\cup_* Y) \to \Def^{<0}(X\cup_* Y)$$ is smooth. But this contradicts Theorem \ref{pre-thm}.
\end{proof}

\begin{stwr}\label{nonreducedhash}
	Assume that $X$ and $Y$ are uncleavable schemes satisfying Setting \ref{set-gor}. Then
	$\Def^{<0}(X\#Y)$ (and also	$\Def^{<0}(X\#\emptyset \cup_* Y)$) is non-reduced.
\end{stwr}
\begin{proof}
	In Theorem \ref{thm-negative}, we have shown that $\Def^{<0}(X\#Y)$ is a closed subfunctor of ${\Def^{<0}(X\cup_* Y)}$, cut out by one equation in degree $d$. Since, by Corollary \ref{non-reduced-union}, $\Def^{<0}(X\cup_* Y)$ is non-reduced, we get a hypersurface of a reduced scheme.
	Since the length of $X\times Y$ is greater than $2$ (since neither is a closed point), a subscheme of colength $1$ is non-reduced as well.
	Therefore, we get the statement.
\end{proof}

\section{Applications to the Gorenstein locus}
In this section, we return to studying of the Gorenstein locus of the Hilbert scheme of points to show applications of the obtained results.
Let us recall, from Definition \ref{HGor-defi}, that $\HG{m}{n}$ denotes the Gorenstein locus of the Hilbert scheme of $m$-points on $\A^n$. 

\begin{twr}\label{app1}
	 	Let us take a pair of uncleavable Gorenstein finite schemes
	 $$X=\Spec \Apolar(F),$$
	 $$Y=\Spec\Apolar(G),$$
	 $F\in P_x$ and $G\in P_y$, $F,G$ homogeneous polynomials of degree $d\geq 3$,
	 such that
	 $$ (T^1_X)_{<-1}=(T^1_Y)_{<-1}=0$$
	 is satisfied.
	 We add two conditions
	 \begin{itemize}
	 	\item 	The following second cotangent modules vanish	$$T^2(\O_X,\kk)_{<-d}=T^2(\O_Y,\kk)_{<-d}=0.$$
	 	\item The ideals $I_X=\Ann(F)\subset S_x$ and $I_Y=\Ann(G)\subset S_y$ are generated in {the~degrees~${\leq d-1}$}.
	 \end{itemize}
	 \vskip 0.2 cm
	 Then, the negative spike $\Hilb^{<0}_{[X\#Y]}$ has the smooth type of a family over the product of negative spikes of $X$ and $Y$, such that the special fibre is $$\frac{\O_X\otimes \O_Y}{f-g},$$ where $f\circ F=1$ and $g\circ G=1$.
\end{twr}
\begin{proof}
	By Theorem~\ref{thm-negative} the smooth type of $\Def^{<0}(X\#Y)$ is as claimed in the statement. By Corollary~\ref{smooth-strict} the map $\Hilb^{<0}_{[X\#Y]} \to \Def^{<0}(X\#Y)$ is smooth, and therefore the functors are smoothly equivalent.
\end{proof}

\begin{stwr}\label{app3}
	If in Theorem~\ref{app1} we assume that $X$ and $Y$ have trivial negative tangents, then 
	the negative spike  $\Hilb^{<0}_{[X\#Y]}$ has the smooth type of 
	$$\frac{\O_X\otimes \O_Y}{f-g}.$$ 
\end{stwr}
\begin{proof}
		By Theorem \ref{thm-negative} we get a smooth map 	 $$\Spec\frac {\O_X\otimes\O_Y}{f-g} \to \Def^{<0}(X\# Y).$$
		By Corollary \ref{smooth-strict} the map $\Hilb^{<0}_{[X\#Y]} \to \Def^{<0}(X\#Y)$, as a consequence $\Spec\frac {\O_X\otimes\O_Y}{f-g}$ is smoothly equivalent to the negative spike.
\end{proof}

\begin{stwr}\label{app4}
	Take $X,Y$ as in Theorem \ref{app1}, and assume that the positive tangent spaces $(T^1_X)_{>0}$, $(T^1_Y)_{>0}$ vanish. Then $\Def(X\# Y)$ is non-reduced and consequently $X\#Y$ gives us a non-reduced point on the Hilbert scheme.
\end{stwr}
\begin{proof}
	By Corollary \ref{nonreducedhash}, we obtain non-reducedness of $\Def^{<0}(X\#Y)$.
	We get a Cartesian diagram with a section as presented below:
	$$\begin{tikzcd}
		\Def^{<0}(X\#Y) \arrow[d] \arrow[r] & \Def^{\leq 0}(X\#Y) \arrow[d]        \\
		* \arrow[r]                         & \Def^{0}(X\#Y). \arrow[u, bend right]
	\end{tikzcd}
	$$
	Lemma \ref{lemma-section2} implies that $ \Def^{\leq 0}(X\#Y)$ is non-reduced.
	Since the positive tangent space vanishes, by Lemma~\ref{smooth-nonstrict}, we get that $\Def^{\leq 0}(X\#Y)$ is smoothly equivalent to $\Hilb^{\leq 0}_{X\#Y}$ and by Lemma~\ref{abb-smooth} we get that the latter is smoothly equivalent to the Hilbert scheme at this point.
\end{proof}

\begin{stwr}\label{main-result}
	 For every $n_1,n_2 \geq 6$, but neither is equal to $7$, then $\HG{2n_1+2n_2+2}{n_1+n_2}$ is not reduced.
\end{stwr}
\begin{proof}
	By Lemma \ref{cubic-positive} for a cubic in $n$ variables with the Hilbert function equal $(1,n,n,1)$, the positive tangent space vanishes. If we take two such cubics $F,G$ in $n_1$ and $n_2$ variables respectively, then the connected sum of spectra of their apolar algebras $X\#Y$ has length ${2n_1+2n_2+2}$.\\
	By Corollary~\ref{app4}, if such cubics satisfy Setting~\ref{set-gor}, we get the statement. \vskip 0.2 cm
	Since the moduli space of cubics in $n$ variables is irreducible, and all conditions are open, we only need to prove that there is a cubic for every condition and every integer $n\geq 6$ which is not $7$.
		\vskip 0.2 cm
	The condition of vanishing of $T^2(\O_X,\kk)$ is open, see Lemma \ref{opent2}, is satisfied by cubics admitting only linear syzygies; there are such cubics. Similarly, the condition that generators of the ideal are in degree $2$ is open, by semicontinuity of Betti numbers.
	\vskip 0.2 cm
	The condition on vanishing of $T^1_X$ in degrees $<-1$ is satisfied by very general cubics, which do exist for $n>7$ and $n=6$, see \cite[Lemma 6.21]{Iarrobino1999},\cite[Theorem 2]{Robert-Gorenstein}. 
		This finishes the proof.
	\end{proof}
	\begin{remark}
		The non-reduced locus described in Corollary \ref{main-result} is \textbf{not} open in irreducible component of $\HG{2n_1+2n_2+2}{n_1+n_2}$ as the component $(1,n_1+n_2,m_1+n_2,1)$ is generically smooth, see Proposition \ref{cubic-component}.
	\end{remark}
	
		One can ask about limitations of our method. We argue that number of variables has to be at least $6$. This is due to the fact that for $n<6$ or $n=7$, the apolar algebras to cubics in $n$ variables are smoothable. Hence we cannot apply Corollary \ref{pointed-nonreduced}.
		

	\subsection{Application to cactus schemes}
	In this subsection we apply our results to the study of the cactus schemes.
	
	\begin{defi}
		Let $Y$ be a subscheme of $\P^N$. \emph{The scheme theoretic linear span} $\langle Y \rangle $ is the smallest linear subspace $\P^q\subset \P^N$ containing $Y$.
	\end{defi}
	
	For a given projective scheme $X\subset \P^N$, \emph{the $r$-th cactus scheme} as a set is given as 
	$$ \cactus{r}{X}:= \overline{\bigcup\{\langle R \rangle |R\subset X,\ R \text{ is a subscheme of degree at most $r$ }\}} \subset \P^N,$$
	it admits a scheme structure given by minors of catalecticant matrices, first introduced in \cite[Theorem 1.5]{BuBu-cactus}. Recently, it was shown that this structure agrees with more geometrical one, based on familiars, see \cite[Definition 4.3, Theorem 1.1]{jabuhanieh-familiar}. \vskip 0.2 cm
	 Let $\nu_d$ be the $d$-th Veronese embedding. To apply our main result, Corollary~\ref{main-result}, to the study of cactus scheme, we need a result that compares the singularities of the Gorenstein locus with the cactus scheme.
	\begin{twri}[{\cite[Theorem 2]{jabuhanieh-familiar}}]
		Let $r\geq 1$ and $d\geq 2r$.
		There exists the following diagram of smooth maps
		$$\begin{tikzcd}
			& Z \arrow[rd,"sm",two heads] \arrow[ld] &          \\
			\cactus{r}{\nu_d(\P^n)}  &                                    & \HG{r}{n},
		\end{tikzcd}
		$$
		where $Z \to \cactus{r}{\nu_d(\P^n)}$ is an open embedding with $\cactus{r-1}{\nu_d(\P^n)} $ being the complement of the image, and $Z\to \HG{r}{n}$ is smooth and surjective.
	\end{twri}

\begin{stwr}\label{final-app}

	Let $n_1,n_2\geq 6$ but neither equal to $7$, and $d\geq 4n_1+4n_2+4$. The cactus scheme $\cactus{2n_1+2n_2+2}{\nu_d(\P^{n_1+n_2})}$ is non-reduced.
\end{stwr}
\begin{proof}
	By Theorem above it is enough to show the non-reducedness for $\HG{2n_1+2n_2+2}{n_1+n_2}$, but this is the statement of Corollary~\ref{main-result}.
\end{proof}


\bibliographystyle{alpha}
\bibliography{references}
\bibliographystyle{alpha}
	\end{document}